\documentclass[11pt,fleqn]{amsart}

\usepackage[usenames,dvipsnames,condensed]{xcolor}
\usepackage{caption}
\usepackage{subcaption}
\usepackage{mathtools}

\usepackage{graphicx}
\graphicspath{ {images/} }
\usepackage{amssymb}
\usepackage{tikz-cd}
\usepackage{tikz}
\usepackage{amsmath}
\usepackage{dsfont}
\usepackage{shuffle}
\usepackage{afterpage}
\usepackage{setspace}
\graphicspath{ {/Users/emanuele/Documents} }
\usepackage[utf8]{inputenc}

\usepackage{mathrsfs}
\usepackage[all]{xy}
\usetikzlibrary{knots}
\usepackage[lite]{amsrefs}
\usepackage{bbm}

\newtheorem{theorem}{Theorem}[section]

\newtheorem{lemma}[theorem]{Lemma}
\newtheorem{proposition}[theorem]{Proposition}

\newtheorem{definition}[theorem]{Definition}

\newtheorem{example}[theorem]{Example}

\newtheorem{remark}[theorem]{Remark}

\allowdisplaybreaks

\newcommand\Z{\mathbb{Z}}

\title[Higher Arity Self-Distributivity and Cohomology]{Higher Arity Self-Distributive Operations in Cascades  and  their Cohomology}

\author[Elhamdadi]{Mohamed Elhamdadi} 
\address{Department of Mathematics, 
	University of South Florida, Tampa, FL 33620, U.S.A.} 
\email{emohamed@math.usf.edu} 

\author[Saito]{Masahico Saito} 
\address{Department of Mathematics, 
	University of South Florida, Tampa, FL 33620, U.S.A.} 
\email{saito@usf.edu} 

\author[Zappala]{Emanuele Zappala} 
\address{Department of Mathematics, 
	University of South Florida, Tampa, FL 33620, U.S.A.} 
\email{zae@mail.usf.edu}

\begin{document}

\maketitle

\begin{abstract}
We investigate constructions of higher arity self-distributive operations,
and give relations between cohomology groups corresponding to operations of different arities. 
For this purpose we introduce the notion of mutually distributive $n$-ary operations
generalizing those for the binary case, and define a cohomology theory labeled by these operations.
A geometric interpretation in terms of framed links is described, with the scope of providing algebraic background 
of constructing $2$-cocycles for framed link invariants.
This theory is also studied in the context of symmetric monoidal categories.
Examples from Lie algebras, coalgebras and Hopf algebras are given.
\end{abstract}

\tableofcontents

\section{Introduction}

Self-distributivity of binary operations and its cohomology theories, motivated by knot theory, have been studied extensively in recent years, see for example \cite{CES, CEGS,CJKLS, CKS,EN,NosakaBook}. In particular, in \cite{CJKLS} the authors have introduced a (co)homology theory of quandles  and utilized the $2$-cocycles with abelian coefficients to define a link invariant, called {\it cocycle invariant}. In \cite{CEGS}, it has been given a non-abelian version of these results, and it has been shown that the cocycle invariant is a {\it quantum} invariant, based on previous work of Gra\~na, see \cite{Gra}. Mutually distributive binary operations were investigated in \cite{Jozef} in which the author proposed a general framework to study homology of distributive structures.  Ternary self-distributivity, which is a natural generalization of the binary case, and its cohomology theory have been studied in \cite{EGM,Maciej}. In this article we consider higher arity self-distributivity and its cohomology, which provide an algebraic background suitable to the definition of  $2$-cocycle framed link invariants. 
 
We generalize the notion of mutual distributive operations given in \cite{Jozef} to $n$-ary operations.  
We show that composing mutually distributive operations results in new higher arity self-distributive operations 
(Proposition~\ref{prop:TmTn}).
Specifically,  for  $m$-ary and $n$-ary  self-distributive operations  $W_m$ and $W_n$ on $X$, 
respectively, under the condition called   {\it mutual distributivity} defined in Section~\ref{sec:nary}, 
the composition   $W: X^{m+n-1} \rightarrow X$ defined  by 
$$
W( x , {\bf y}, {\bf z}) = W_n ( W_m (x, {\bf y}) , {\bf z} ) 
$$
is shown to be
an $(m+n-1)$-ary self-distributive operation. This procedure, in the specific, although important, case of mutually distributive binary operations, is particularly proficuous to describe colorings of framed link diagrams by means of ternary operations, paving the way for the possibility of introducing an analogue of the cocycle invariant in the context of framed links. We defer the study of such an invariant to a subsequent work.

We generalize both the $n$-ary distributive homology \cite{EGM} and homology of distributive sets \cite{Jozef}
to mutually distributive sets of general $n$-ary operations (Definition~\ref{def:labchain}).
The relation between this chain complex and the $n$-ary operations that result from mutually distributive sets
by composition as in Definition~\ref{LDchainmaps},
is  given in the form of a chain map. 

We also present  constructions, called {\it doubling}, for mutually distributive sets, 
that are similar to the composition 
introduced
in Section~\ref{sec:nary} but defined on the product $X \times X$.
A geometric interpretation,
uses
 parallel strings 
and allows us to relate colorings of diagrams of links and colorings of diagrams of framed links. The cohomological counterpart of this procedure, also described in this paper, is expected to provide the algebraic context for the definition of a framed link cocycle invariant.
The relation between the doubling and the composition in Section~\ref{sec:nary}, 
as well as its implications to cohomology are discussed.

Higher arity self-distributivity is also investigated in the context of symmetric monoidal categories.  
We define the notion of $n$-ary self-distributive object in a symmetric monoidal category, providing therefore a higher arity version of the work in \cite{CCES}. In particular, we show how to produce ternary self-distributive objects in the category of vector spaces, starting from binary self-distributive objects (Theorem~\ref{co-dbl}). 
 Specifically, let $(X,\Delta)$ be a comonoid in a (strict) symmetric monoidal category $\mathcal{C}$ (e.g. a coalgebra in the category of vector spaces). Let $q:X\boxtimes X\longrightarrow X$ be a morphism such that $(X,q)$ is a binary self-distributive object in $\mathcal{C}$. Then the pair $(X, T)$, where $T = q(q\boxtimes \mathbbm{1})$, defines a ternary self-distributive object in $\mathcal{C}$. The construction defines a functor $\mathcal{F}:\mathcal{BSD}\rightarrow \mathcal{TSD}$, from the category of binary self-distributive objects in $\mathcal{C}$, to the category of ternary self-distributive objects in $\mathcal{C}$.
This procedure of internalization of higher order self-distributivity indeed produce interesting examples of self-distributive objects among coalgebras.  Examples from Lie algebras, coalgebras and Hopf algebras are given. 
As in \cite{CCES} this categorical version is expected to be related to the Yang-Baxter equation in tensors of vector spaces through Hopf algebras. 

The article is organized in the following manner. 
Section \ref{sec:pre} gives the basics of binary and ternary racks with examples.  In Section \ref{sec:nary} we introduce the higher arity case and 
show that composing mutually distributive operations results in new higher arity self-distributive operations.
In Section~\ref{sec:labeledhom} we define a cohomology theory that generalizes those given in \cite{EGM} and  \cite{Jozef} to mutually distributive operations in various arities. 
A chain map that relates this chain complex and  the one for the operation resulted from composition 
is given in Section \ref{sec:chain-map}. 
 In Section \ref{sec:dble} we introduce functors (which we call doubling) in  the binary and ternary mutually distributive rack categories.
Section~\ref{sec:TtoR} shows the passages between 
 binary mutually distributive racks and 
 ternary self-distributive racks. 
We exhibit a direct construction of ternary cocycles from binary cocycles.  
We close the circle of functors relating binary and ternary operations by introducing a construction that brings back from ternary to binary. 
The relations among these functors are also discussed 
and 
a geometric interpretation is given.  Section \ref{Inter} is devoted to the development of a 
categorical point of view of $n$-ary self-distributivity, 
extending the previous results of \cite{CCES}. We define self-distributive objects in symmetric monoidal categories and construct examples in the category of vector spaces. 
We describe a procedure to obtain higher order self-distributive operations from Lie algebras.  Appendix~\ref{Ap:Lie} deals with some detailed computations related
to Lie algebras. In Appendix~\ref{augmented} we introduce a higher order analogue of augmented rack that enables us to produce
Hopf algebra versions of group theoretic examples, such as the heap operation.

\section{Preliminary}\label{sec:pre}

\subsection{Basics of Racks}
We review, for the convenience of the reader, some basic definitions of shelves, racks and quandles and give a few examples. This material can be found, for example, in \cite{Jozef,CKS,FR, EN,NosakaBook}.
\begin{definition}
{\rm

A {\it shelf }  $X$ is a set with a binary operation $(a, b) \mapsto a
* b$ such that
for any $a,b,c \in X$, we have
$ (a * b) * c=(a* c)* (b* c). $

If, in addition,  the maps $R_y: x \mapsto x*y$ are bijections of $X$, for all $y\in X$, then 
$(X, *)$ is called a {\it rack}.

A {\it quandle}  is an idempotent ($a* a =a,$  for all $a \in X$) rack. 

%
%
%

}
\end{definition}


\begin{example}
{\rm
The following are typical examples of quandles: 
\begin{itemize}
	\item 

A group $G$ with
conjugation as operation: $a * b = b^{-1} a b$,
denoted by $X=$ Conj$(G)$, is a quandle. 

\item
A group $G$ with
the operation $a*b=ba^{-1}b $ is a quandle called the {\it core} quandle.


\item
Any $\Lambda (={\Z }[t, t^{-1}])$-module $M$ is
a quandle with $a* b=ta+(1-t)b$, for $a,b \in M$, and is called
an {\it  Alexander  quandle}. 

\item 

For any group $G$, and an automorphism $f\in Aut(G)$, the operation $x*y = f(xy^{-1})y$ defines a quandle structure on $G$, usually referred to as {\it generalized Alexander quandle}. 
\end{itemize}
}
\end{example}

A rack {\it homomorphism} $f: (X, *) \rightarrow (X', *')$  is a map satisfying 
$f(x*y)=f(x)*' f(y)$ for all $x, y \in X$.
The category of racks is denoted by ${\mathcal R}$.

 Let $(X,*)$ be a  rack and $A$ be an abelian group.  A function $\phi: X \times X \rightarrow A$ is said to be a (rack) $2$-\textit{cocycle} 
 if for all $x,y,z \in X$, the following holds
$$
\phi(x,y)+\phi(x*y,z)=\phi(x,z)+ \phi(x*z,y*z).
$$

\begin{lemma}[\cite{CENS}] \label{lem:abext}
Let $(X, *)$ be a rack, $A$ be an abelian group, and $\phi: X \times X \rightarrow A$ be a $2$-cocycle.
Define an operation $*$ on $X \times A$ by 
$$
(x, a) * (y, b) := (x*y , a + \phi(x,y) )
.$$
Then $(X \times A, *)$ is a rack.
\end{lemma}

Rack and quandle 2-cocycles have been constructed from extensions \cite{CENS}, 
polynomial expressions \cite{AmeSai,Mochi}, determinants \cite{Nosaka}, and computer calculations \cite{Vend}.

	\subsection{Ternary distributive structures}\label{subsec:ternary}
	
	Ternary racks and quandles were investigated in \cite{EGM,Green,Maciej} and generalized further in \cite{CEGM}.  Here we review the basics of ternary racks and give some examples.
	
\begin{definition}\label{ternarydef}  {\rm 

		Let $(X,T)$ be a set equipped with a ternary operation $T: X\times X \times X \rightarrow X$.  The operation  $T$ is said to be {\it (right) distributive} if it satisfies the following condition
		for all  $ x,y,z,u,v \in X $,
		$$
		T(T(x,y,z),u,v)=T(T(x,u,v),T(y,u,v),T(z,u,v)). 
		$$ 
	}
\end{definition}

In this paper we will consider distributivity from the right. 
%
%

\begin{definition}\label{ternaryquandledef}  {\rm 
		Let $T: X \times X \times X \rightarrow X $ be a  ternary distributive operation on a set $X$. 
		If for all $a,b \in X$, the map $R_{a,b}:X \rightarrow X$ given by $R_{a,b}(x)=T(x,a,b)$ is invertible,  
		then $(X,T)$ is said to be \emph{ternary rack}.  

	}
\end{definition}

\begin{example}\label{ex:ternary}
{\rm

	The following constructions are found in \cite{EGM}.
\begin{itemize}
\item
	Let $(X,*)$ be a rack and define a ternary operation on $X$ by $T(x,y,z)=(x * y)* z$, for all $x,y,z \in X$.  It is straightforward to see that $(X, T)$ is a ternary rack.  Note that in this case $R_{a,b}=R_b \circ R_a$. We will say that this ternary rack is induced by a (binary) rack.
	
In particular, 	if $(X,*)$ is an Alexander quandle with $x*y=tx+(1-t)y$, then the ternary rack coming from $X$ has the operation $$T(x,y,z)=t^2x+t(1-t)y+(1-t)z. $$

\begin{sloppypar}
\item
	Let $M$ be any ${\Lambda}$-module where
	${\Lambda}=\mathbb{Z}[t^{\pm 1},s]$. The operation $T(x,y,z)=tx+sy+(1-t-s)z$ defines a ternary rack structure on $M$.  We call this an {\it affine} ternary rack.
	
In particular, 	consider $\Z_8$ with the ternary operation $T(x,y,z)=3x+2y+4z$.  This affine ternary rack given in \cite{EGM}  is not induced by an Alexander quandle structure as described in the preceding item since $3$ is not a square in $\Z_8$.
\end{sloppypar}

\item
	Any group $G$ with the ternary operation $T(x,y,z)=xy^{-1}z$  gives a ternary rack.  This operation is well known and called a {\it heap} (sometimes also called a groud) of the group $G$.
\end{itemize}

}
\end{example}

A morphism of ternary racks is a map $f: ( X, T) \rightarrow (X',T')$ such that $$f(T(x,y,z))=T'(f(x),f(y),f(z)).$$  A bijective ternary rack endomorphism is called ternary rack automorphism.
We denote by ${\mathcal T}$ the category of ternary racks.

Let $(X, T)$ be a ternary rack and $A$ be an abelian group.  A function $\psi: X \times X \times X \rightarrow A$ is said to be a \textit{ternary} $2$-\textit{cocycle} 
 if for all $x,y,z, u, v \in X$, the following hold
\begin{eqnarray*}  
\lefteqn{ \psi(x, y, z) + \psi(T(x, y, z), u, v) }\\
& = & \psi(x, u, v) 
	+  \psi(T(x, u, v), T(y, u, v), T(z, u, v)).
\end{eqnarray*}
This equation is motivated by the following lemma, which is verified by calculations.
\begin{lemma}\label{lem:ternary-abext}
Let $(X, T)$ be a ternary rack and $A$ be an abelian group.  Let $\phi: X \times X \times X \rightarrow A$ be a map.
	The set $X \times A$ with the ternary operation given by 
	$$T((x,a), (y,b),(z,c))=(T(x,y,z), a+ \psi(x,y,z))$$
	 is a ternary rack if and only if the map $\phi$ satisfies the following ternary $2$-cocycle condition
	\begin{eqnarray*}  
	\lefteqn{ \phi(x, y, z) + \phi(T(x, y, z), u, v)}\\
	&  = & \phi(x, u, v) 
	+  \phi(T(x, u, v), T(y, u, v), T(z, u, v)).
	\end{eqnarray*}
\end{lemma}

\begin{figure}[htb]
    \begin{center}
   \includegraphics[width=2in]{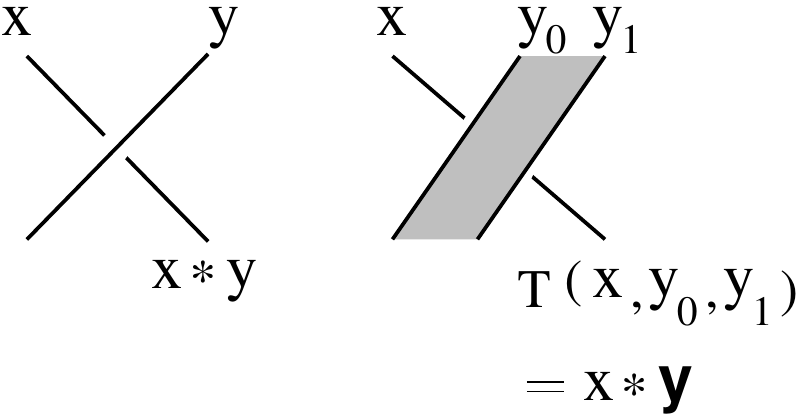}\\ 
    \caption{Diagrammatic representations of a binary (left) and ternary (right) operations 
    }\label{fig:diag1}
    \end{center}
\end{figure}

For a ternary distributive operation $T$ on $X$, we also use the notation
$$x * {\bf y} := T(x, y_0, y_1), $$
 where ${\bf y}=(y_0, y_1)$.
Although strictly speaking $T(x, y_0, y_1)$ is not equal to $T(x, (y_0, y_1))$, no confusion is likely to arise by this convention.
Furthermore, for ${\bf x}=(x_0, x_1)$, we use the notation
${\bf x} * {\bf y}$ to represent $$(x_0 * {\bf y}, x_1 * {\bf y}) = (T(x_0, y_0, y_1), T(x_1, y_0, y_1)). $$
In this notation the ternary distributivity can be written as 
$$ ( x * {\bf y}) * {\bf z}= (x * {\bf z} ) *  ({\bf y} * {\bf z} )  $$ 
in analogy to the binary case.

Figure~\ref{fig:diag1} depicts diagrammatic representations of binary and  ternary  operations, 
on the left and on the right, respectively.  See \cite{CKS}, for example, for more details on diagrammatics for racks and their knot colorings.

We also recall the definition of homology of ternary racks \cite{EGM}. 
Define first $C_n(X)$ to be the free abelian group generated by $(2n+1)$-tuples $(x_0,x_1, \ldots, x_{2n})$ of elements of a ternary rack $(X,T)$. Define the differentials
 $\partial_n : C_n(X) \longrightarrow C_{n-1}(X)$ as:
\begin{eqnarray*}
	\lefteqn{\partial _n (x_0, x_1, \ldots, x_{2n})}\\
	&=& \sum_{i=1}^{n}(-1)^i[(x_0,\ldots, \widehat{x_{2i-1}}, \widehat{x_{2i}}, \ldots, x_{2n})\\
	&& - (T(x_0,x_{2i-1},x_{2i}), \ldots, T(x_{2i-2},x_{2i-1}, x_{2i}, ),\widehat{x_{2i-1}}, \widehat{x_{2i}},\ldots, x_{2n} )].
\end{eqnarray*}
\begin{definition}
{\rm

	The $n^{th}$ homology group of the ternary rack $X$ is defined to be: 
	\begin{center}
		$ H_n(X) = {\rm ker} \partial_n/ {\rm im} \partial_{n+1}$.
	\end{center}
	
	}
\end{definition}

By dualizing the chain complex given above, we get a cohomology theory for ternary racks.

\begin{remark}
	{\rm
		Similar definitions give a homology and a cohomology theory for higher arity self-distributive operations. 
	}
\end{remark}

\section{Compositions of $n$-ary self-distributive operations } \label{sec:nary} 

In this section we generalize the notion of mutual distributive operations found in \cite{Jozef} to $n$-ary operations.
The vector notation for ternary operations is directly generalized to the $n$-ary ones:
Let $(X, W)$ be an $n$-ary distributive set.
Let  ${\bf y}=(y_1, \ldots, y_{n-1}) \in X^{n-1}$. Then the operation
$W: X^n \rightarrow X$ is denoted by $W(x, y_1, \ldots, y_{n-1})=W(x, {\bf y})$.
An  $n$-ary operation  is also denoted by $x * {\bf y} := W(x, {\bf y})$.
Here the extra parentheses caused by the vector notation is ignored, i.e., 
for ${\bf y}=(y_1, \ldots, y_{n-1})$ and ${\bf z}=(z_1, \ldots, z_{n-1})$, the concatenation 
$( {\bf y}, {\bf z} )$ or simply $ {\bf y}, {\bf z} $ denotes $(y_1, \ldots, y_{n-1}, z_1, \ldots, z_{n-1})$.
Furthermore, for ${\bf x} =(x_1, \ldots, x_{m}) \in X^m$ and ${\bf y} \in X^{n-1}$, denote 
$(W(x_1, {\bf y}), \ldots, W(x_m, {\bf y}) )$ by $W({\bf x} , {\bf y})$ or ${\bf x} * {\bf y}$. 

\begin{definition}\label{def:TmTn}
	{\rm
		Let $W_m$ and $W_n$ be $m$-ary and $n$-ary distributive operations on $X$, respectively.
		The two operations $W_m$ and $W_n$ are called {\it mutually distributive} if they satisfy 
		\begin{eqnarray*}
			W_n ( W_m (x, {\bf y}) ,  {\bf z} ) &=&  W_m (  W_n (x, {\bf z}), W_n ({\bf y}, {\bf z} ) ) \\
			W_m ( W_n (x, {\bf u}) ,  {\bf v} ) &=&  W_n (  W_m (x, {\bf v}), W_m ({\bf u}, {\bf v} ) )
		\end{eqnarray*}
		for all $x \in X$, ${\bf y}, {\bf v} \in X^{m-1}$ and ${\bf z}, {\bf u} \in X^{n-1}$.
	}
\end{definition}

\begin{example} \label{ex:halftriv}
	{\rm
		Let $(X, *_X)$, $(Y, *_Y)$ be racks. 
		Define $*_0, *_1$ on $X \times Y$, respectively, by 
		$(x_0, y_0) *_0 ( x_1, y_1)= (x_0 *_X x_1, y_0)$ and 
		$(x_0, y_0) *_1 ( x_1, y_1)= (x_0 , y_0 *_Y y_1)$.
		Then computation shows that $( *_0, *_1)$ are mutually distributive.
	}
\end{example}

\begin{example}\label{ex:powerrack}
	{\rm
		The following construction appears in \cite{IshiiIwakiri} and 
		provides examples of mutually distributive rack operations.
		Denote by $*^n$ the rack operation on $X$ defined by $n$-fold leftmost product
		$x *^n y = (\cdots (x*y)*y)* \cdots * y$. 
		Then $*_0=*^m$ and $*_1=*^n$ are mutually distributive for positive integers $m$ and $n$. 
		
		More generally, the following appears in \cite{IIJO,Jozef}. Let $X$ be a group, and 
		let $f_0 , f_1 \in {\rm Aut}(X)$ be mutually commuting group automorphisms. 
		Let $*_\epsilon$ be the  generalized Alexander quandles with respect to 
		$f_\epsilon$ for $\epsilon=0,1$. 
		Thus $x*_\epsilon y=(x y^{-1})^{f_{\epsilon} } y$, where the action is denoted in exponential notation.
		Then computations show that $*_0$ and $*_1$ are mutually distributive.
		
		
	}
\end{example}

There are mutually distributive operations with different arities, as the following example shows.

\begin{example}
	{\rm
		Let $X$ be a module over $\Z [ u^{\pm 1}, t^{\pm 1}, s]$ and $*$, $T$ be affine binary and ternary rack
		operations, respectively, defined  by 
		\begin{eqnarray*} 
			x * y &=& ux + (1-u)y , \\
			T(x,y, z) &=& t x + s y + (1-t-s)z .
		\end{eqnarray*}
		Then computations show that $*$ and $T$ are mutually distributive.
	}
\end{example}

\begin{remark}
	{\rm
		We note that for a group $G$, 
		the core binary operation  ($x*y=yx^{-1}y$)
		and the ternary operation heap ($x \ \hat{*}\ (y_0, y_1)=
		x y_0^{-1}y_1 $) satisfy 
		$( x * y ) \ \hat{*}\  {\bf z} = ( x  \ \hat{*}\  {\bf z} ) * ( y  \ \hat{*}\  {\bf z} )$ but not
		$( x \ \hat{*}\  {\bf y}  ) *  z = ( x  * z  ) \ \hat{*} \  ( {\bf y}  * z  )$.
		
	}
\end{remark}

Next, we show that composing mutually distributive operations results in new higher arity self-distributive operations.

\begin{proposition}\label{prop:TmTn}
	Let $W_m$ and $W_n$ be mutually distributive $m$-ary and $n$-ary distributive operations on $X$.
	Then $W: X^{m+n-1} \rightarrow X$ defined by 
	$$
	W( x , {\bf y}, {\bf z}) = W_n ( W_m (x, {\bf y}) , {\bf z} ) 
	$$
	is an $(m+n-1)$-ary distributive operation.
\end{proposition}

\begin{proof} We establish the equality
	$$
	W(W (x, {\bf y} , {\bf z}), {\bf u}, {\bf v})= W(W (x, {\bf u} , {\bf v}),W({\bf y},
	{\bf u}, {\bf v}), W({\bf z}, {\bf u}, {\bf v})).
	$$
	We replace $W_n(x, {\bf y})$ by the notation $ x *_n {\bf y}$.  Thus we have 
	$$
	W (x, {\bf y} , {\bf z}):=(x *_m {\bf y})*_n {\bf z}.
	$$
	Then we compute
	\begin{eqnarray*}
		\lefteqn{W(W (x, {\bf y} , {\bf z}), {\bf u}, {\bf v})}\\
		&=& [[(x *_m {\bf y}) *_n {\bf z}]*_m {\bf u}]*_n {\bf v}  \\
		&=& ([(x *_m {\bf y}) *_m {\bf u}]*_n [{\bf z}*_m {\bf u}])*_n {\bf v}\\
		&=& [[(x *_m {\bf u}) *_m ({\bf y}*_m {\bf u})]*_n ({\bf z}*_m {\bf u})]*_n {\bf v}\\
		&=& [[(x *_m {\bf u}) *_m ({\bf y}*_m {\bf u})]*_n {\bf v}] *_n  [({\bf z}*_m {\bf u})*_n {\bf v}]	\\
		&=& [[ (x *_m {\bf u}) *_n {\bf v} ]  *_m  [ ({\bf y}*_m {\bf u}) *_n {\bf v}] ] *_n  [({\bf z}*_m {\bf u})*_n {\bf v}]	\\
		&=& W(W (x, {\bf u} , {\bf v}),W({\bf y},
		{\bf u}, {\bf v}), W({\bf z}, {\bf u}, {\bf v})),
	\end{eqnarray*}
	where the second and the fifth equalities follow from the mutual distributivity of $*_m$ and $ *_n$.  This concludes the proof. 
\end{proof}

\begin{remark}
	{\rm 
	Let $(X, *_0, *_1)$ be mutually distributive binary operations. 
Let $T$ be the TSD operation defined in 
Proposition~\ref{prop:TmTn}. 
Then it is written as 
$T(x, y, z) = (x *_0 y ) *_1 z $ for $x,y,z\in X$.
We note that 
		the two ternary structures $T(x,y,z)=(x*_0 y ) *_1 z$ and $T'(x,y,z)=(x*_1 y ) *_0 z$ may not be isomorphic in general as the following example shows.
		
		Consider the set $\mathbb{Z}_3$ with the two binary operations $x*_0 y=x$ and $x*_1y=2y-x$.  The induced ternary structures $T(x,y,z)=(x*_0 y ) *_1 z$ and $T'(x,y,z)=(x*_1 y ) *_0 z$ are not isomorphic. 
		In fact, if $f: (\mathbb{Z}_3, T) \rightarrow  (\mathbb{Z}_3, T')$ is an isomorphism then for all $x, y, z$ in $\mathbb{Z}_3$, we have $ f(T(x,y,z))=T'(f(x),f(y),f(z)).$  Then $f(2z-x)=2f(y)-f(x)$. One obtains then a contradition, for example, by setting $x=z=0$. 
		
	}
\end{remark}

\begin{definition}
	{\rm
		Let $*_{n_j}$, $j=1, \ldots, k$, be distributive $n_j$-ary operations on $X$ that are pairwise
		mutually distributive. Then we call $(X, \{ *_{n_j}\}_{j=1}^{k})$ 
		a {\it mutually distributive set}.
	}
\end{definition}


%
%
%

\section{Homology of mutually distributive sets}\label{sec:labeledhom}

We generalize both the $n$-ary distributive homology \cite{EGM} and homology of distributive sets \cite{Jozef}
to mutually distributive sets of general $n$-ary operations
as follows. The relation between this chain complex and the $n$-ary operations that result from mutually distributive sets
as in Proposition~\ref{prop:TmTn} will be given in the next section in the form of chain map.

\begin{definition}\label{def:labchain}
	{\rm
		
		Let  $(X, \{ *_{n_j}\}_{j=1}^{k})$ be a mutually distributive set.
		Let $\vec{ \epsilon}=(\epsilon_1, \ldots, \epsilon_{n-1})$ be a vector such that $\epsilon_i\in \{n_j \}_{j=1}^{k}$
		for $i=1, \ldots, n-1$. 
		Let chain groups $C^{\vec{ \epsilon}}_n(X)$ be defined by the free abelian group generated by 
		tuples ${\bf x} = (x_0, ( {\bf x}_1,  \epsilon_1) , \ldots, ( {\bf x}_{n-1}, \epsilon_{n-1}) )$.
		Define $C_n(X)=\oplus_{\vec{ \epsilon}} \ C^{\vec{\epsilon}}_n(X)$ where the direct sum ranges over all 
		possible vectors $\vec{\epsilon}$.
		Define the differential 
		$\partial^{\vec{ \epsilon}}_n: C^{\vec{ \epsilon}}_n(X) \rightarrow C_{n-1}(X)$ by 
		\begin{eqnarray*}
			\partial^{\vec{ \epsilon} }_n ({\bf x} )
			& = & \sum_{i=1}^{n-1} (-1)^i [ ( x_0 *_{\epsilon_i} {\bf x}_i , ({\bf x}_1 *_{\epsilon_i} {\bf x}_i , \epsilon_1), 
			\ldots,  (  {\bf x}_{i-1} *_{\epsilon_i} {\bf x}_i , \epsilon_{i-1}), \widehat{( {\bf x}_i , \epsilon_{i})}, \\
			& & \hspace{2in} ( {\bf x}_{i+1} , \epsilon_{i+1}),
			\ldots, ( {\bf x}_{n-1}, \epsilon_{n-1}) ) \\
			& &\hspace{1in}  - 
			(x_0, ( {\bf x}_1,  \epsilon_1) , \ldots, \widehat{( {\bf x}_i , \epsilon_{i})}, 
			\ldots, ( {\bf x}_{n-1}, \epsilon_{n-1}) ) ] ,
		\end{eqnarray*}
		and let $$\partial_n = \sum_{\vec{ \epsilon}} \partial^{\vec{ \epsilon}}_n: C_n(X)\rightarrow C_{n-1}(X).$$
	}
\end{definition} 

\begin{lemma}\label{ChainCmplx}
	Let $(X, \{ *_{n_j}\}_{j=1}^{k})$ be a mutually distributive set. Then the sequence
	$(C_n(X), \partial_n)$ defines a chain complex.
\end{lemma}

\begin{proof}
	We define, for each vector $\vec{ \epsilon}$ and $i=1, \ldots, n-1$, linear maps 
	\begin{eqnarray*}
		\partial^{i\vec{\epsilon}}_n({\bf x} ) 
		&=&  [ ( x_0 *_{\epsilon_i} {\bf x}_i , ({\bf x}_1 *_{\epsilon_i} {\bf x}_i , \epsilon_1), 
		\ldots,  ({\bf x}_{i-1} *_{\epsilon_i} {\bf x}_i , \epsilon_{i-1}), \widehat{( {\bf x}_i , \epsilon_{i})}, \\
		& & \hspace{2in} ( {\bf x}_{i+1} , \epsilon_{i+1}),
		\ldots, ( {\bf x}_{n-1}, \epsilon_{n-1}) ) \\
		& &\hspace{1in}  - 
		(x_0, ( {\bf x}_1,  \epsilon_1) , \ldots, \widehat{( {\bf x}_i , \epsilon_{i})}, 
		\ldots, ( {\bf x}_{n-1}, \epsilon_{n-1}) ) ].
	\end{eqnarray*}
	Therefore by definition, $\partial^{\vec{\epsilon}}_n = \sum_i (-1)^i \partial^{i\vec{\epsilon}}_n$. It is enough to show now that the maps $\partial^{i\vec{\epsilon}}_n$ satisfy the pre-simplicial complex relation: $\partial^{i\vec{\epsilon}}_{n-1}\partial^{j\vec{\epsilon}}_n = \partial^{j\vec{\epsilon}}_{n-1}\partial^{(i+1) \vec{\epsilon}}_n$ for each $n\in\mathbb{N}$ whenever $j<i$. 
	
	\begin{sloppypar}
		Fix a vector $\vec{\epsilon} = (\epsilon_1, \ldots , \epsilon_{n-1})$ and consider an element $(x_0, (x_1,\epsilon_1),\ldots , (x_{n-1},\epsilon_{n-1})) \in C^{\vec{ \epsilon}}_n(X)$. Then we have: 
		\begin{eqnarray*}
			&&\partial^{i}_{n-1}\partial^{j}_n (x_0, ({\bf x}_1,\epsilon_1),\ldots , ({\bf x}_{n-1},\epsilon_{n-1}))  =\\
			&& ((x_0*_{\epsilon_j}{\bf x}_j )*_{\epsilon_{i+1}}{\bf x}_{i+1},( ({\bf x}_1 *_{\epsilon_j}{\bf x}_j )*_{\epsilon_{i+1}}{\bf x}_{i+1}, \epsilon_1 ), \ldots , \widehat{( {\bf x}_j , \epsilon_{j})}, \\ 
			&& ({\bf x}_{j+1}*{\bf x}_{i+1},\epsilon_{j+1}), \ldots, \widehat{( {\bf x}_i , \epsilon_{i})}, \ldots , ({\bf x}_{n-1},\epsilon_{n-1})) \\
			&& - (x_0, ({\bf x}_1 , \epsilon_1 ), \ldots , \widehat{( {\bf x}_j , \epsilon_{j})}, \\ 
			&& ({\bf x}_{j+1},\epsilon_{j+1}), \ldots, \widehat{( {\bf x}_{i+1} , \epsilon_{i+1})}, \ldots , ({\bf x}_{n-1},\epsilon_{n-1})).
		\end{eqnarray*}
	\end{sloppypar}
	
	On the other hand we have:
	\begin{eqnarray*}
		&&\partial^{j\vec{\epsilon}}_{n-1}\partial^{(i+1)\vec{\epsilon}}_n (x_0, ({\bf x}_1,\epsilon_1),\ldots,({\bf x}_{n-1},\epsilon_{n-1})) =\\
		&&((x_0*_{\epsilon_{i+1}}{\bf x}_{i+1})*_{\epsilon_{j}}({\bf x}_j*_{\epsilon_{i+1}}{\bf x}_{i+1}),(({\bf x}_1*_{\epsilon_{i+1}}{\bf x}_{i+1})*_{\epsilon_{j}}({\bf x}_j*_{\epsilon_{i+1}}{\bf x}_{i+1} ), \epsilon_1),\\ 
		&&\ldots \widehat{( {\bf x}_j , \epsilon_{j})}, ({\bf x}_{j+1}*_{\epsilon_{i+1}}{\bf x}_{i+1}, \epsilon_{j+1}), \ldots , \widehat{( {\bf x}_{i+1} , \epsilon_{i+1})}, \ldots , ({\bf x}_{n-1},\epsilon_{n-1})) \\
		&& - (x_0, ({\bf x}_1 , \epsilon_1 ), \ldots , \widehat{( {\bf x}_j , \epsilon_{j})}, \\ 
		&& ({\bf x}_{j+1},\epsilon_{j+1}), \ldots, \widehat{( {\bf x}_{i+1} , \epsilon_{i+1})}, \ldots , ({\bf x}_{n-1},\epsilon_{n-1})),
	\end{eqnarray*}
	where we have used the vector notation introduced in Section \ref{sec:nary}. The two quantities are equal, in virtue of the property of mutual distributivity of the set $\{ *_{n_j}\}_{j=1}^{k}$.  
\end{proof}

\begin{definition}
	{\rm 
		
		The chain complex defined by Definition \ref{def:labchain} and the homology that it induces will be called labeled chain complex and labeled homology and will be denoted $C^L_\bullet(X)$ and $H^L_\bullet(X)$, respectively.
	}
\end{definition}
\begin{figure}[htb]
	\begin{center}
		\includegraphics[width=4.5in]{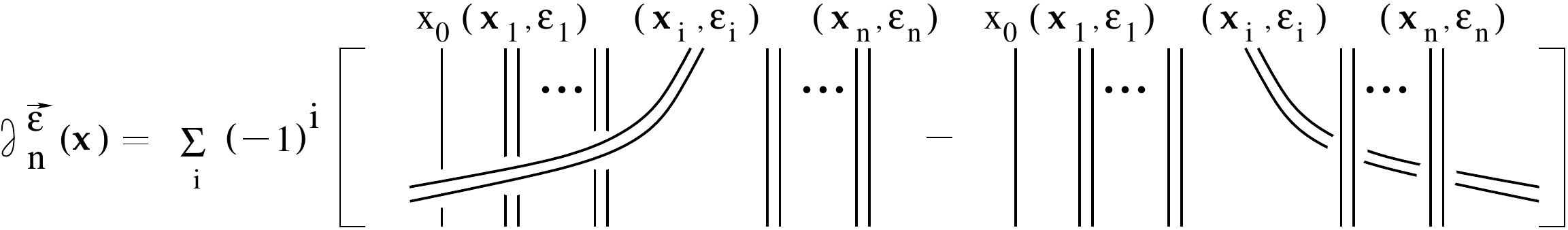}\\
		\caption{Curtain diagram representing chain maps}\label{fig:ribboncurtain}
	\end{center}
\end{figure}

\begin{remark}
	{\rm 
		
		The chain complex in Definition \ref{def:labchain} has a diagrammatic interpretation as in Figure \ref{fig:ribboncurtain}. In particular, the mutual distributivity condition takes the same form as in the curtain homology of \cite{PW}. 
	}
\end{remark}
\begin{remark}
	{\rm
		We observe that if $(X,\{*_{n_j} \})$ is a mutually distributive set, then $C^L_\bullet(X)$ contains the standard self-distributive complexes relative to each $*_{n_j}$ as subcomplexes. 
	}
\end{remark}

%
%
%
%


\begin{remark}
	{\rm
		
		The multiplication on binary operations considered in \cite{Jozef} can be directly generalized to $n$-ary operations
		as follows.
		Given a nonempty set $X$, let ${\rm Dist}_{M}(X)$ denotes the set of all $n$-ary mutually distributive operations on $X$.
		Define the following multiplication on ${\rm Dist}_{M}(X)$:
		$$
		(W \cdot W')(x,{\bf y}):=W(W'(x,{\bf y}),{\bf y})
		$$
		for all $x \in X$ and ${\bf y} \in X^{n-1}$.
		Then it is straightforward to see that the  multiplication defined above makes ${\rm Dist}_{M}(X)$ into a 
		monoid 
		with identity $W_0$  given by $W_0(x,{\bf y} )=x,$ for all $ x \in X$ and ${\bf y}\in X^{n-1}$. 
		
		For example, let $(X,T)$  be a  ternary rack.
		Define, inductively, 
		$$T^n(x, y_0 , y_1)=T ( T^{n-1} ( x, y_0, y_1) , y_0, y_1 ). $$
		Then  $(X, T^n)$ is a ternary distributive set for all positive integer $n$.
		
	}
\end{remark}



\begin{remark}\label{rmk:labcohomology}
	{\rm
		For a given abelian group $A$, we obtain a labeled cochain complex with coefficients in $A$, upon dualizing the chain complex in Definition \ref{def:labchain}. We will write $C^n_L(X;A)$ and $H^n_L(X;A)$ to indicate the labeled $n^{\rm{th}}$ cochain and cohomology groups with coefficients in $A$, respectively. 
	}
\end{remark}

\section{Chain maps under $n$-ary compositions  }\label{sec:chain-map}

In this section we show that the cohomology of an operation obtained by composing mutually distributive operations as in Proposition~\ref{prop:TmTn} and the cohomology of the operations themselves are related to the labeled cohomology of Definition~\ref{def:labchain} via chain maps as follows.
The algebraic and geometric motivations and significance of the chain maps are explained later in this section.

\begin{definition}\label{LDchainmaps}
	{\rm 
		Let $(X,*_0,*_1)$ be a distributive set, where $*_0$ and $*_1$ are operations of arity $k$ and $k'$, respectively. Call $W$ the $(k+k'-1)$-ary corresponding self-distributive operation. We define chain maps ${\mathcal F}_{\sharp, n} : C_n^W(X) \longrightarrow C_n^L (X)$, from the $(k+k'-1)$-ary cochain complex relative to $W$, to the chain complex defined by Lemma \ref{ChainCmplx} for $n=1,2,3$. Explicitly: 
		\begin{eqnarray*}
			{\mathcal F}_{\sharp, 1} &=& \mathbf{1}\\
			{\mathcal F}_{\sharp, 2}(x,{\bf y}_0,{\bf y}_1) &=& (x,{\bf y}_0)_0 + (x*_0 {\bf y}_0, {\bf y}_1)_1 \\
			{\mathcal F}_{\sharp, 3}(x,{\bf y}_0,{\bf y}_1,{\bf z}_0,{\bf z}_1) &=&  (x,{\bf y}_0,{\bf z}_0)_{00} + (x*_0{\bf z}_0,{\bf y}_0*_0{\bf z}_0,{\bf z}_1)_{01} + \\
			&&  (x*_0{\bf y}_0,{\bf y}_1,{\bf z}_0)_{10} + ((x*_0{\bf y}_0)*_0{\bf z}_0,{\bf y}_1*_0{\bf z}_0,{\bf z}_1)_{11},
		\end{eqnarray*}
		where we put the labels as a subscript and ${\bf y}_i$, ${\bf z}_i$ are vectors of appropriate lengths, according to the conventions explained in Section~\ref{sec:nary}.
	}
\end{definition}

%

\begin{definition}\label{def:labcohmap}
	{\rm
		Let $(X, *_0, *_1)$ be a mutually distributive racks and $A$ be an abelian group. Let  ${\mathcal F}^{\sharp, n}: C^n_L(X,A) \rightarrow C^n_W(X,A)$ for $n=2,3$ be the maps obtained from ${\mathcal F}_{\sharp, n}$ by dualization.
	}
\end{definition}

\begin{theorem}\label{lem:labcohmaps} 
	For $n=2,3$ the maps ${\mathcal F}_{\sharp, n}$ define chain maps.
	Therefore they define induced homomorphisms ${\mathcal F}_{*, n} : H^W_n(X,A) \rightarrow H_n^L(X,A)$ in homology and 
	${\mathcal F}^{*, n} : H^n_L(X,A) \rightarrow H^n_W(X,A)$ in cohomology.
\end{theorem}

\begin{proof}
	We prove the statement in the case of two binary mutually distributive operations $*_0$ and $*_1$, resulting in a ternary self-distributive operation $T$. The general case being an application of the same procedure with vector notation.
	For a ternary $2$-chain $(x,y_0,y_1)$ we have:
	\begin{eqnarray*}
		\lefteqn{ 
			\partial {\mathcal F}_{\sharp, 2}(x,y_0,y_1) =}\\
		& &  -(x*_0y_0) + (x) - ( (x*_0y_0)*_1y_1 ) + (x*_0y_0 )= \partial_T  (x,y_0,y_1). 
	\end{eqnarray*}	
	By direct computation, we also have: 
	\begin{eqnarray*}
		\lefteqn{ 
			{\mathcal F}_{\sharp, 2}\partial_T (x,y_0,y_1,z_0,z_1) 
		}\\
		&= & (x,z_0)_0 + (x*_0 z_0,z_1)_1 - (T(x,y_0,y_1),z_0)_0 \\
		&& - (T(x,y_0,y_1)*_0z_0,z_1)_1 - (x,y_0)_0 - (x*_0y_0,y_1)_1 \\
		&& + ( T(x,z_0,z_1), T(y_0,z_0,z_1))_0 \\
		& & + (T(x,z_0,z_1)*_0T(y_0,z_0,z_1), T(y_1,z_0,z_1))_1.
	\end{eqnarray*}  
	On the other hand, the following holds: 
	\begin{eqnarray*}
		\lefteqn{
			\partial {\mathcal F}_{\sharp, 3} (x,y_0,y_1,z_0,z_1) }
		\\
		&= &(x,z_0)_0 - (x *_0 y_0,z_0)_0 - (x,y_0)_0 \\
		&& + (x*_0z_0,y_0*_0z_0)_0 + (x*_0z_0, z_1)_1 - ((x*_0z_0)*_0(y_0*_0z_0),z_1)_1\\
		&& - (x*_0z_0,y_0*_0z_0)_0 + (T(x,z_0,z_1),T(y_0,z_0,z_1))_0  \\
		&& + (x*_0y_0,z_0)_0 - (T(x,y_0,y_1),z_0)_0 \\
		&& - (x*_0y_0,y_1)_1 + ((x*_0y_0)*_0z_0,y_1*_0z_0)_1\\
		&& + ((x*_0y_0)*_0z_0,z_1)_1 - (((x*_0y_0)*_0z_0)*_1(y_1*_0z_0),z_1)_1 \\
		&& - ((x*_0y_0)*_0z_0,y_1*_0z_0)_1 + (((x*_0y_0)*_0z_0)*_1z_1,T(y_1,z_0,z_1))_1 .
	\end{eqnarray*}	
	The two quantities can be seen to be equal, making use of the identity:
	\[
	T(x,z_0,z_1)*_0T(y_0,z_0,z_1) = ((x*_0y_0)*_0z_0)*_1z_1.
	\]	
	Therefore we obtain	${\mathcal F}_{\sharp, 2}\partial_T = \partial {\mathcal F}_{\sharp, 3}$, which concludes the proof of the first statement.  The second statement follows easily from the first one by standard arguments  in homological algebra.
\end{proof}

\begin{remark}
{\rm
 Let $(X,*_0,*_1)$ be a mutually distributive rack and let $C^2_L(X;A)$ be the second labeled cochain group with coefficients in $A$, then 
 the labeled $2$-cocycle conditions corresponding to $\delta^{(01)} \psi =0 $ and $\delta^{(10)} \psi =0$ take the following form
 	\begin{eqnarray*}
 	\phi_0(x,y) + \phi_1(x*_0 y,z) &=&  \phi_1(x,z) + \phi_0(x*_1 z,y *_1 z),  \label{one-zero}\\
 	\phi_1(x,y) + \phi_0(x*_1y,z) &=& \phi_0(x,z) + \phi_1(x*_0z,y *_0 z).  
 \end{eqnarray*}
 }
 \end{remark}
 
 \begin{definition}\label{def:mutual2cocy}
 {\rm 
We call a pair $(\phi_0,\phi_1)$ satisfying the preceding  equations \textit{mutually distributive}. 
}
\end{definition}

Observe  that $(\phi_0,\phi_1)$ being a labeled $2$-cocycle means that it is mutually distributive, $\phi_0$ is a $2$-cocycle for the operation $*_0$ and $\phi_1$ is a $2$-cocycle for $*_1$. 

\begin{remark}\label{rem:2to3}
{\rm 
Let $(X, *_0, *_1)$ be mutually distributive binary operations
and $T$ be 
the ternary self-distributive operation defined in 
Proposition~\ref{prop:TmTn}. 
Let $A$ be an abelian group.
By Theorem~\ref{lem:labcohmaps}, 
${\mathcal F}^{\sharp, 2} (\phi_0, \phi_1) = \psi $ is a ternary $2$-cocycle in $C^2_T (X, A)$
for a labeled 2-cocycle $(\phi_0, \phi_1)$. 
The explicit form of the ternary $2$-cocycle
\begin{eqnarray*}\label{TernCocy}
	\psi(x,y,z):=\phi_0(x,y) + \phi_1(x*_0 y,z).
\end{eqnarray*}
}
\end{remark}

\begin{remark}\label{mutualrack2cocy}
 {\rm
 The case of mutually distributive binary operations whose composition gives a ternary operation is of particular interest to us since 
this is the algebraic counterpart of a diagrammatic doubling procedure
 particularly adapt to interpret colorings of framed 
 tangles by ternary racks. The ternary $2$-cocycles resulting from Theorem~\ref{lem:labcohmaps} 
 can therefore be used to define cocycle invariants for framed tangles. 
 This construction of cocycles corresponds to 
 those in \cite{IshiiIwakiri} for handlebody-links.
 
From this geometric point of view, we present a direct, geometric 
 proof that $\psi$ in Remark~\ref{rem:2to3} satisfies the ternary 2-cocycle condition.
 We  show
\begin{eqnarray*}  
	\lefteqn{ \psi(x, y_0,  y_1) + \psi(T(x, y_0, y_1), z_0, z_1)  = }\\
	& &  \psi(x, z_0, z_1)  
	+  \psi(T(x, z_0, z_1), T(y_0, z_0, z_1), T(y_1, z_0, z_1)).
\end{eqnarray*}
The computations  below are aided by diagrams shown in Figure~\ref{fig:2cocy}, where each equality is represented by a type III Reidemeister move. In the figure and the computations below, underlines highlight those terms to which the cocycle condition is applied. 
\begin{eqnarray*}
	{\rm LHS}
	&=& \phi_0(x,y_0) +
	\underline{  \phi_1(x *_0 y_0, y_1)  + \phi_0((x *_0 y_0 )*_1 y_1), z_0)  } \\
	& & \quad + \phi_1((x*_0 y_0)*_1 y_1)*_0 z_0 , z_1)   \nonumber \\
	&=& 
	\underline{  \phi_0(x,y_0) +  \phi_0(x *_0 y_0, z_0) } + \phi_1((x*_0 y_0)*_1 y_1)*_0 z_0 , z_1)  \\
	& & \quad + \phi_1( (x*_0 y_0)*_0 z_0,  y_1*_0 z_0 ) \\
	&=& \phi_0(x,z_0) + \phi_0(x *_0 z_0,y_0 *_0  z_0)
	+ \underline{ \phi_1((x*_0 y_0)*_0 z_0)*_0 z_0 , y_1*_0 z_0 ) } \\
	& & \quad + \underline{ \phi_1((( x*_0 y_0) *_1 y_1) *_0 z_0 , z_1 ) } \\
	&=&  \phi_0(x,z_0) + \underline{\phi_0(  x*_0 z_0, y_0 *_0 z_0) +  \phi_1((x*_0 y_0)*_0 z_0) , z_1 ) }\\
	& & \quad + \phi_1((x*_0 y_0)*_0 z_0)*_1 z_1 , (y_1*_0 z_0 )*_1 z_1 ) \\
	&= & 
	\phi_0(x,z_0) +  \phi_1(x*_0 z_0, z_1 )\\
	& & \quad +  \phi_1((x*_0 y_0)*_0 z_0)*_1 z_1 , (y_1*_0 z_0 )*_1 z_1 ) \\
	& & 
	\quad +  \phi_0((x*_0 z_0)*_0 z_0)*_1 z_1 , (y_1*_0 z_0 )*_1 z_1 )   \quad = \quad {\rm RHS}
\end{eqnarray*}
}
\end{remark}

\begin{figure}[htb]
	\begin{center}
		\includegraphics[width=5in]{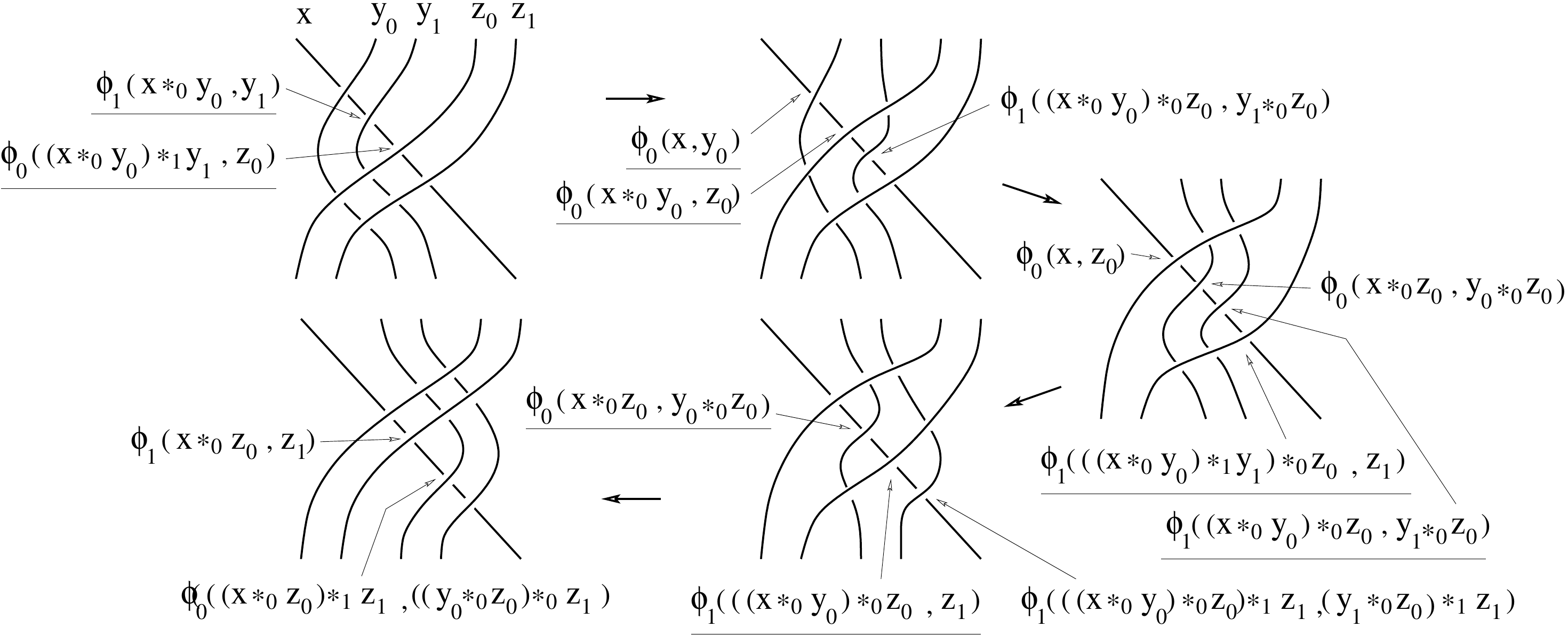}\\
		\caption{Diagrammatic proof of 2-cocycle conditions }\label{fig:2cocy}
	\end{center}
\end{figure}

\begin{example} 
	{\rm
		Let $(X, *_X)$, $(Y, *_Y)$ be racks, and 
		$( *_0, *_1)$ be  mutually distributive operations defined
		on $X \times Y$  in Example~\ref{ex:halftriv}.
		Let $\phi_X$ and $\phi_Y$ be 2-cocycles 
		of  $(X, *_X)$ and  $(Y, *_Y)$, respectively.
		Define 2-cocycles of $X \times Y$ corresponding to $*_0$, $*_1$, respectively,  by 
		$\phi_0( (x_0, y_0), (x_1, y_1) ) = \phi_X(x_0, x_1)$ and 
		$\phi_1( (x_0, y_0), (x_1, y_1) ) = \phi_Y(y_0, y_1)$. 
		Then computations show that $(\phi_0, \phi_1)$ are mutually distributive.
	}
\end{example}
\begin{example}
	{\rm 
		The following construction, found in \cite{IshiiIwakiri}, 
		provides examples of mutually distributive 2-cocycles.
		Let $(X, *)$ be a rack, $\phi: X \times X \rightarrow A$ be a 2-cocycle, 
		and $(E=X \times A, \tilde{*})$ be the corresponding extension.
		Recall that  $*^n$ denotes  the $n$-fold leftmost product
		$x *^n y = (\cdots (x*y)*y)* \cdots * y$. 
		Then  the function  $\phi_n$ defined by 
		$$
		\phi_n(x,y)= \phi(x, y) + \phi(x*y, y) + \cdots + \phi(x *^{n-1} y, y)
		$$
		is a $2$-cocycle. 
		
		Let $(X, *_0=*^m, *_1=*^n)$ be the mutually distributive rack 
		defined in Example~\ref{ex:powerrack}, 
		and let $\phi_m$, $\phi_n$ be 2-cocycles  defined above.
		Then  $\phi_m$ and $\phi_n$ are mutually distributive.
		This is seen by the diagrammatic interpretation of parallel strings.
		
}\end{example}

\section{The doubling functor} \label{sec:dble} 

In this section we describe a construction called {\it doubling}, 
that is similar to the composition defined in Section~\ref{sec:nary} but defined on the product $X \times X$.
A diagrammatic interpretation is to take  parallel strings 
and provides a method of producing cocycle  invariants for framed links by means of ternary cohomology. 
The relation between the doubling and the composition in Section~\ref{sec:nary}
as well as implications to cohomology are discussed in the next section.

\subsection{Doubling binary operations}

%
%

\begin{lemma}\label{lem:dbleR}
	Let $(X, *_0, *_1)$ be mutually distributive racks.
	Define 
	the operation
	for $(x_0, x_1), (y_0, y_1) \in X \times X$ by
	$$ (x_0, x_1)  *  (y_0, y_1) := (  ( x_0 *_0  y_0) *_1 y_1, ( x_1 *_0  y_0) *_1 y_1 ) .$$
	Then $(X \times X, *)$ is a rack.
\end{lemma}


\begin{figure}[htb]
    \begin{center}
   \includegraphics[width=1in]{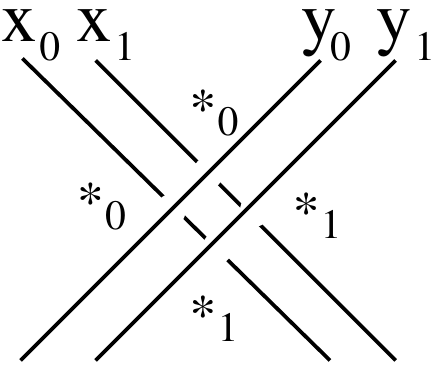}\\ 
    \caption{Diagrammatic representation of doubling    }\label{fig:diag2}
    \end{center}
\end{figure}

A diagrammatic representation of the preceding lemma is depicted in Figure~\ref{fig:diag2}, 
and the computations in its proof are facilitated by the corresponding type III Reidemeister move
with doubled strings.

\begin{definition}
	{\rm
		Let ${\mathcal R}_M$ be the category defined as follows.
		The objects consist of $(X, *_0, *_1)$, where $X$ is a set and $(*_0, *_1)$ is mutually distributive.
		For objects $(X, *_0, *_1)$ and $(X', *_0', *_1')$, a morphism 
		$f$ is a map $f: X \rightarrow X'$ that is a rack morphism for both $(*_0, *_0')$ and 
		$(*_1, *_1')$.
	}
\end{definition}

We observe that if $f: X \rightarrow X'$ is a morphism in the sense of this definition, then $f$ will automatically respect the mutual  distributivity. Specifically, simple computations imply the following.

\begin{lemma}\label{lem:RtoTfunctor}
	If $f: (X, *_0, *_1) \rightarrow (X', *_0', *_1')$ is a morphism in ${\mathcal R}_M$, 
	then it holds that 
	$$
	f(   (x*_0 y ) *_1 z ) =( f(x) *_1' f( z )) *_0' ( f(y) *_1' f(z) ) .
	$$
\end{lemma}

Computations also show the following.

\begin{lemma}\label{lem:dbleRmap}
	Let $(X, *_0, *_1) $ and $(X', *_0', *_1')$ be two mutually distributive racks,
	and $(X \times X, *)$ and $(X' \times X', *')$ be racks as in Lemma~\ref{lem:dbleR}. 
	If $f: (X, *_0, *_1) \rightarrow (X', *_0', *_1')$ is a morphism in ${\mathcal R}_M$, then the map
	$F: (X \times X, *) \rightarrow (X' \times X', *')$ defined by $F( x, y) = (f(x), f(y))$ 
	is a rack morphism.
\end{lemma}

\begin{definition}
	{\rm 
		The functor ${\mathcal D}_R$ from ${\mathcal R}_M$ to the category ${\mathcal R}$ of binary racks 
		defined on objects by 
		${\mathcal D}_R ( X, *_0, *_1)  = (X \times X, *)$ through
		Lemma~\ref{lem:dbleR}  and on morphisms by ${\mathcal D}_R (f) = f\times f$ through Lemma~\ref{lem:dbleRmap}, is called the {\it doubling} functor.
	}
\end{definition}

\begin{remark}
	{\rm
	The functor ${\mathcal D}_R$ is injective on objects and 
	morphisms, but not surjective on either.
}
\end{remark}


A direct computation gives the following lemma.

\begin{lemma}\label{lem:mutualext}
Let $(X, *_0, *_1)$ be a mutually distributive rack, and $( \phi_0, \phi_1 ) $ be mutually distributive rack 2-cocycles. 
Let $(E, \tilde{*}_\epsilon )$ be abelian extensions of $(X, *_\epsilon )$ with respect to $\phi_\epsilon$,
$$ (x, a) \ \tilde{*}_\epsilon \  (y, b)=(x *_\epsilon y, a + \phi_\epsilon (x, y) ) $$
for $\epsilon=0, 1$. 
Then $(E, \tilde{*}_0, \tilde{*}_1)$ is a mutually distributive rack.
\end{lemma}

\begin{theorem}\label{thm:doubleRcocy}
	Let $(X, *_0, *_1)$ and $(X \times X, *)$ be as described in Lemma~\ref{lem:dbleR}.
	Let $\phi_0, \phi_1$ be rack 2-cocycles of $(X, *_0)$ and $(X, *_1)$, respectively, that satisfy the mutually distributive 
	rack 2-cocycle condition. 
	Then 
	$$\phi( (x_0, x_1), (y_0, y_1) ) 
	= \phi_0 ( x_0, y_0)   + \phi_1 ( x_0 *_0 y_0, y_1)  + \phi_0 (x_1, y_0) + \phi_1 (x_1 *_0 y_0 , y_1) $$
	is a rack 2-cocycle of $(X \times X, *)$. This assignment induces a well defined map $\Theta: H^2_L(X)\longrightarrow H^2_R(X\times X)$, where the subscript $R$ indicates the binary rack cohomology. 
\end{theorem}

A proof will be given at the end of Section~\ref{sec:TtoR}. 
The right-hand side corresponds to Figure~\ref{fig:diag2}.
We call $\phi$ the {\it doubled rack 2-cocycle}.

%
%
%
%

\subsection{Doubling ternary  operations} 

In this subsection, we give a doubling construction for ternary racks.
The condition required for this construction differs from the mutual distributivity and defined as follows.

\begin{definition}\label{Termutuality}
{\rm
Let $T_0$ and $T_1$  be two ternary operations on a set $X$.
We say that $T_0$ and $T_1$ are {\it compatible} if they satisfy 
\begin{eqnarray*}
\lefteqn{
T_0 ( T_0 (x_0, y_0, y_1), z_0, z_1)  } \\
&=& 
T_0 ( T_0 (x_0, z_0, z_1),  T_0 (y_0, z_0, z_1),  T_1 (y_1, z_0, z_1)  ),  \\
\lefteqn{
T_1 ( T_1 (x_1, y_0, y_1), z_0, z_1 ) } \\
&=& 
T_1 ( T_1 (x_1, z_0, z_1),  T_0 (y_0, z_0, z_1),  T_1 (y_1, z_0, z_1)  ) .
\end{eqnarray*}
}
\end{definition}

A diagrammatic representation of the compatibility is depicted in Figure~\ref{fig:dbleT}. Observe that it corresponds to type III Reidemeister move for ribbons.

\begin{figure}[htb]
    \begin{center}
   \includegraphics[width=3in]{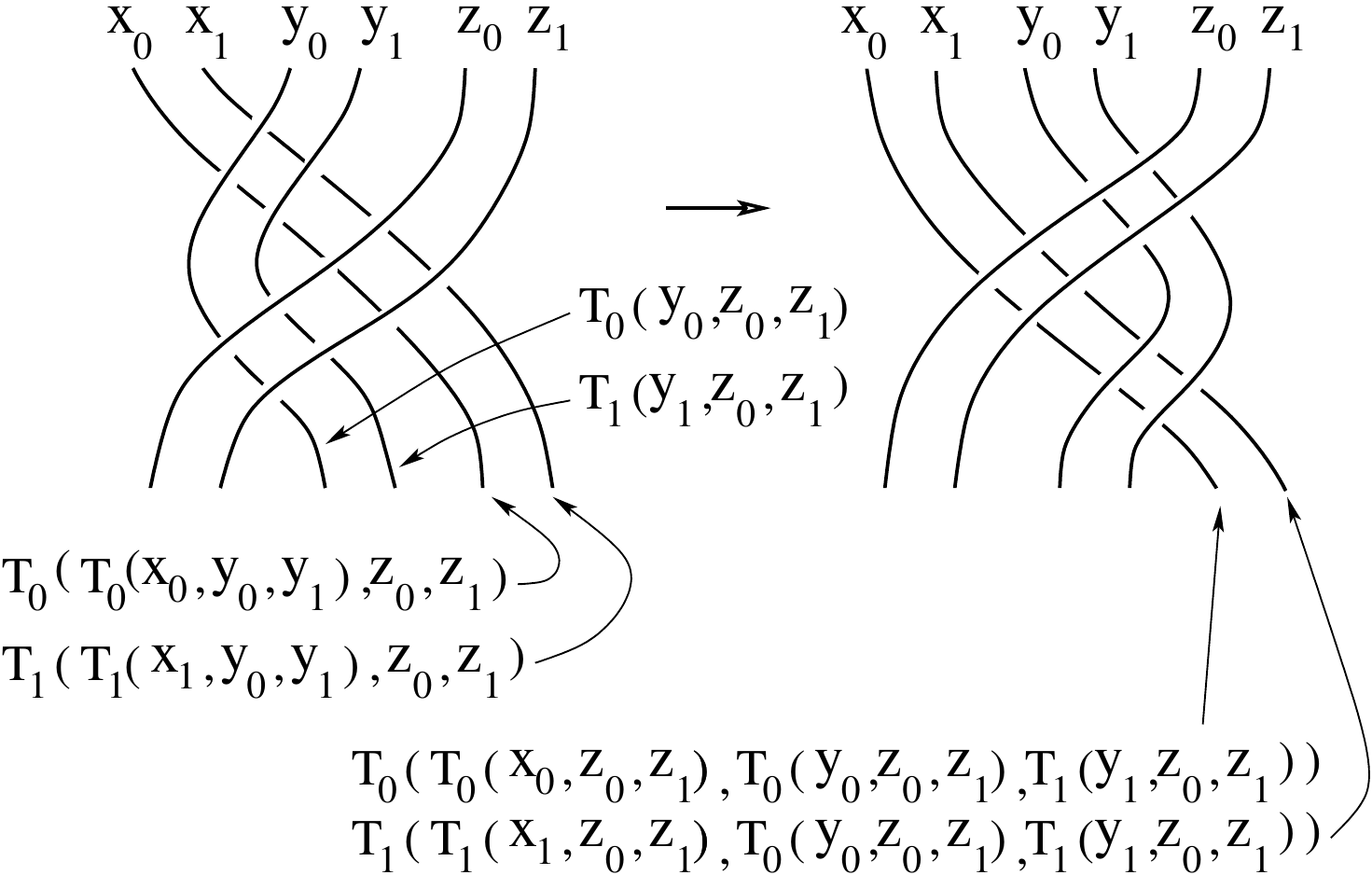}\\ 
    \caption{Diagrammatic representation of  compatible ternary rack operations }\label{fig:dbleT}
    \end{center}
\end{figure}

\begin{example}{\rm 
	Consider a ${\Lambda}$-module $M$ where
	${\Lambda}=\mathbb{Z}[t^{\pm 1},t'^{\pm 1},s,s']$. The following two ternary operation $T_0(x,y,z)=tx+sy+(1-t-s)z$ and $T_1(x,y,z)=t'x+s'y+(1-t'-s')z$ are compatible if and only if the following conditions hold 
		\begin{center} $$	\begin{array}{ll} 
	\begin{cases}
		(1-t-s)(t'-t)=0\\
		(1-t-s)(s'-s)=0
	\end{cases} \text{and}
	& 
\begin{cases}
		(1-t'-s')(t-t')=0\\
		(1-t'-s')(s-s')=0\ .
	\end{cases}
\end{array} 
$$ \end{center} 

For example, one can choose $M=\mathbb{Z}_{8}$ with $T_0(x,y,z)=3x+2y+4z$ and $T_1(x,y,z)=-x+2y$.
}
\end{example}

\begin{definition} \label{def:TC}
	{\rm
		The category ${\mathcal T}_C$ of compatible 
		 ternary distributive racks is defined as follows. The objects consist of triples $(X,T_0,T_1)$ where $X$ is a set and $(T_0,T_1)$ are  compatible ternary operations on $X$. A morphism between two objects $(X,T_0,T_1)$ and $(Y,T_0',T_1')$ is a map $f: X \rightarrow Y$ which is morphism in the ternary category for both $(T_0,T_0')$ and $(T_1,T_1')$. 
	}
\end{definition}

We observe that if $f: X \rightarrow X'$ is a morphism in the sense of Definition~\ref{def:TC}, then it will automatically respect the mutual ternary distributivity. Specifically, computations imply the following.

\begin{lemma}
If $f: (X, T_0, T_1) \rightarrow (X', T_0', T_1')$ is a morphism in ${\mathcal T}_C$,  
then it holds that 
\begin{eqnarray*}
\lefteqn{  f( T_0 ( T_0 (x_0, y_0, y_1), z_0, z_1 ) ) = } \\
&& 
\hspace{-10pt} T_0' ( T_0'  ( f( x_0) , f( z_0) , f( z_1) ) ,  
T_0' (f( y_0), f( z_0) , f( z_1)  ), 
 T_1'  (f ( y_1) , f( z_0) , f( z_1) )  ),\\
\lefteqn{f(T_1 ( T_1 (x_0, y_0, y_1), z_0, z_1 ) ) 
=}\\
& &
\hspace{-10pt} T_1' ( T_1' (f(x_0), f(z_0), f(z_1)),
 T_0' (f(y_0), f(z_0), f(z_1)), 
T_1' (f(y_1), f(z_0), f(z_1) )  ).
\end{eqnarray*}
\end{lemma}

\begin{theorem}\label{thm:dbleT}
Let $( T_0, T_1)$ be compatible ternary distributive operations on $X$. 
Then $T: X^2 \times X^2 \times X^2 \rightarrow X^2$ defined by 
\begin{eqnarray*}
\lefteqn{ T((x_0, x_1), (y_0, y_1),(z_0, z_1) ) }\\
&=& ( T_0 ( T_0 (x_0, y_0, y_1 ) , z_0, z_1) , T_1 ( T_1 (x_1, y_0, y_1 ) , z_0, z_1) ) 
\end{eqnarray*}
is a ternary distributive operation on $X^2$. 
\end{theorem}

\begin{proof}
It is enough to establish  
\begin{eqnarray*}
\lefteqn{
T ( T (  (x_0, x_1), (y_0, y_1), (z_0 , z_1) ), (u_0, u_1), (v_0, v_1) ) 
} \\
&=&
T (\  T (  (x_0, x_1),  (u_0, u_1), (v_0, v_1) ), \\
& & \hspace{.5in} T( (y_0, y_1), (u_0, u_1), (v_0, v_1) ) , T(  (z_0 , z_1), (u_0, u_1), (v_0, v_1) ) \ ) .
\end{eqnarray*}
A diagrammatic representation of this equality is depicted in Figure~\ref{fig:dbleIII}.
This diagrammatic equality follows from a sequence of moves depicted in Figure~\ref{fig:dbleT}.
Thus calculations are obtained by applications of defining relations of compatibility accordingly.
\end{proof}

\begin{figure}[htb]
    \begin{center}
   \includegraphics[width=3in]{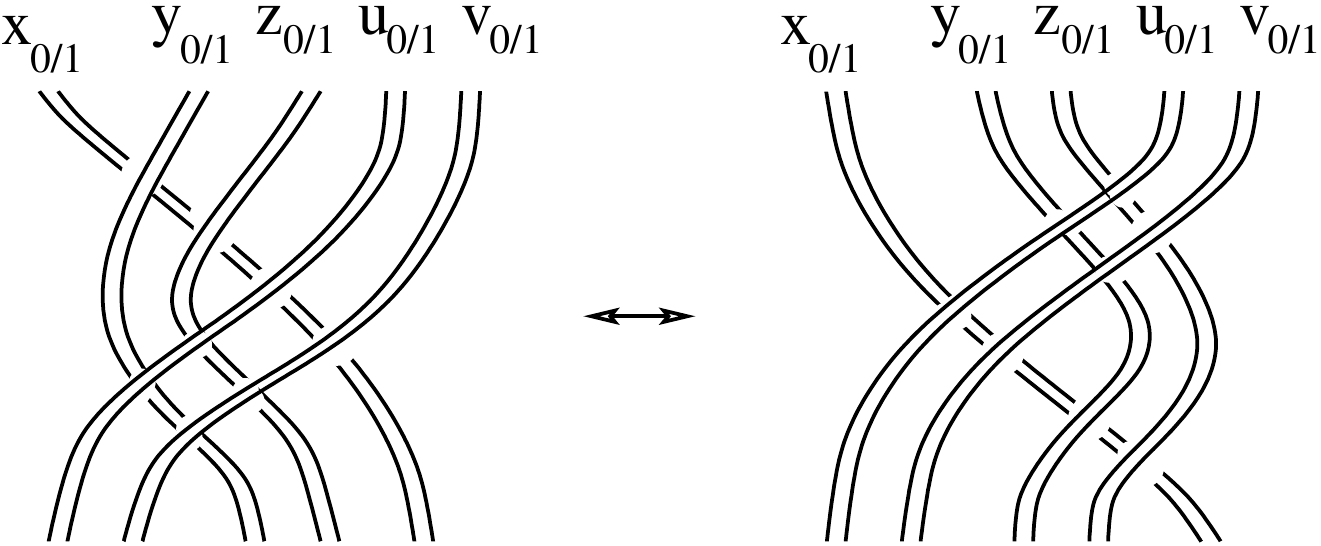}\\ 
    \caption{Diagrammatic representation of Theorem~\ref{thm:dbleT}  }\label{fig:dbleIII}
    \end{center}
\end{figure}


The following is analogous to Lemma~\ref{lem:dbleRmap} and is shown by direct computations.

\begin{lemma}
	Let $(X, T_0, T_1)$ and $(X', T_0', T_1')$  be sets with mutually distributive ternary operations, and 
	$(X \times X, T)$ and $(X' \times X', T')$ be ternary distributive racks constructed in Theorem~\ref{thm:dbleT}.
	If $f: (X, T_0, T_1) \rightarrow (X', T_0', T_1')$ is a morphism in ${\mathcal T}_C$,
	then $F$ defined from $f$ by $f \times f$  is a morphism of ${\mathcal T}_C$. 
\end{lemma}

\begin{definition}
	{\rm 
		We denote the functor  from $ {\mathcal T}_C$  
		 to the category of ternary racks defined on objects  by
		${\mathcal D}_T ( X, T_0, T_1)  = (X \times X, T)$ and on morphisms by ${\mathcal D}_T (f)=f \times f$,
		and call it {\it doubling}.
	}
\end{definition}

\begin{remark}
	{\rm 
		The functor  ${\mathcal D}_T $ is injective on both objects and morphisms, but is not surjective on either.
	}
\end{remark}

\begin{definition} \label{def:compaticocy}
{\rm 

Let $(T_0, T_1)$ be compatible ternary distributive operations on $X$.
Let $\psi_0$, $\psi_1$ be $2$-cocycles with respect to $T_0$ and $T_1$, respectively.
Then the following  are called the compatibility conditions for $\psi_0$ and $\psi_1$:
\begin{eqnarray*}
\lefteqn{	 \psi_0 (x_0, y_0, y_1) + \psi_1 (T_1(x_1, y_0,y_1), z_0, z_1) }\\
&=&	  \psi_1 (x_1, z_0, z_1)+
 \psi_0 (T_0(x_0,z_0, z_1), T_0(y_0, z_0, z_1), T_1(y_1, z_0, z_1)), \\
\lefteqn{  \psi_1 (x_1, y_0, y_1) + 
	 \psi_0 (T_0(x_0, y_0, y_1), z_0, z_1) } \\\
&=&  \psi_0 (x_0, z_0, z_1) + 
	 \psi_1 (T_1(x_0,z_0, z_1), T_0(y_0, z_0, z_1), T_1(y_1, z_0, z_1)). 	
\end{eqnarray*}

}
\end{definition}

\begin{theorem}\label{thm:dbleTcocy}
	Let $(T_0, T_1)$ be compatible ternary distributive  operations on $X$.
	Let $T$ be the doubled ternary operation defined in Theorem~\ref{thm:dbleT}.
Let $\psi_0$, $\psi_1$ be $2$-cocycles with respect to $T_0$ and $T_1$, respectively, that satisfy 
 the compatibility  condition defined in Definition~\ref{def:compaticocy}.
	Then 
	\begin{eqnarray*}
	\lefteqn{
	\psi ( (x_0, x_1), (y_0, y_1), (z_0, z_1)  ) }\\
	&=& \psi_0 ( x_0, y_0, y_1) + \psi_1 (x_1, y_0, y_1) \\
	& & + \psi_0 ( T_0( x_0, y_0, y_1),  z_0, z_1) + \psi_1 ( T_1 ( x_1, y_0, y_1) , z_0, z_1) 
	\end{eqnarray*}
	is a ternary rack 2-cocycle of $(X \times X, T)$. 
\end{theorem}

A proof will be given at the end of Section~\ref{sec:TtoR}.
We call $\psi$ the {\it doubled ternary rack 2-cocycle}.

\section{From  binary racks to ternary racks and back }\label{sec:TtoR} 

In this section we provide relations among constructions of self-distributive operations discussed so far.
To simplify the arguments, we focus on binary and ternary operations. 
Specifically, we observe that the doubling functors  of binary (resp. ternary) operations factor through ternary
(resp. binary) operations.
This main result of the section is stated in Proposition~\ref{prop:factor}.
Furthermore corresponding constructions of 2-cocycles are given,
and proofs of  
Theorems~\ref{thm:doubleRcocy} and \ref{thm:dbleTcocy} are provided at the end of the section.
We start with defining a functor for the construction given in Proposition~\ref{prop:TmTn}.

\begin{definition}\label{def:bintoter}
	{\rm
		The assignment of objects   defined by Proposition~\ref{prop:TmTn} when $W_n$ and $W_m$ are binary operations (hence the obtained $W:= T$ is ternary), is denoted 
		${\mathcal F}(X, *_0, *_1) = (X, T)$. 
		This assignment on objects can be extended on morphisms as the identity, to define a functor ${\mathcal F}: {\mathcal R}_M \rightarrow {\mathcal T}$, from the category of mutually distributive binary racks (see Section~\ref{sec:dble}), to the category of ternary racks, using Lemma~\ref{lem:RtoTfunctor}.
	}
\end{definition}

By definition ${\mathcal F}$ is injective and surjective on morphisms.
Computations give the following.

\begin{lemma}\label{lem:triple}
	Let $\{ *, *_0, *_1 \}$ be a mutually distributive binary set.
	Let $(X, T)={\mathcal F}(X, *_0, *_1)$. Then $\{ *, T \}$ are mutually distributive.
\end{lemma}



Next we define the opposite construction of binary from ternary operations.

\begin{lemma}\label{lem:TtoR}
Let $T_0$ and $T_1$ be a two compatible ternary rack operations. 
Then the binary operation on the cartesian product $X \times X$ defined by  
$$
(x_0, x_1)*(y_0, y_1):=(T_0(x_0, y_0, y_1), T_1(x_1, y_0,y_1)) = (x_0 *_0 {\bf y}, x_1 *_1 {\bf y}) 
$$
gives a rack structure $(X \times X, *)$. 
\end{lemma}


\begin{definition}
{\rm 

The functor defined by Lemma~\ref{lem:TtoR} is denoted by ${\mathcal G}: {\mathcal T}_C \rightarrow {\mathcal R}$, where 
${\mathcal G}(X, T_0, T_1) = (X\times X, *)$ on objects, and ${\mathcal G}(f) = f \times f$ on morphisms. 

}
\end{definition}

Observe that ${\mathcal G}$  is injective on objects and 
on morphisms. 

\begin{proposition}
The functor ${\mathcal G}$ is not surjective on objects. 
\end{proposition}

\begin{proof}
Consider the binary rack structure on $\mathbb{Z} \times \mathbb{Z}$ defined by 
$$(x_0,x_1)*(y_0,y_1) = (x_0+x_1, x_1).$$
 This rack is not in the image of ${\mathcal G}$ since the first entry depends on both $x_0$ and $x_1$. 
\end{proof}



\begin{theorem}\label{thm:TRcocy}
	\begin{sloppypar}
Let $(X, T_0, T_1)$ be an object in ${\mathcal T}_C$, 
and $(X\times X, *) = {\mathcal G}(X, T_0, T_1) $ be as  in Lemma~\ref{lem:TtoR}.
Suppose $\psi_0$ and $\psi_1$ are compatible  ternary $2$-cocycles of respectively $(X,T_0)$ and $(X,T_1)$. 
Then 
$$\phi( (x_0, x_1), (y_0, y_1) ) := \psi_0 (x_0, y_0, y_1) + \psi_1 (x_1, y_0, y_1) $$
defines a 2-cocycle $\phi $ of $(X \times X, *)$. 
\end{sloppypar}
\end{theorem}

\begin{proof}
We  check that $\phi$ satisfies the following equation
\begin{eqnarray*}
\lefteqn{ \phi((x_0, x_1), (y_0, y_1))+ \phi((x_0, x_1)* (y_0, y_1), (z_0,z_1)) } \\
&=& \phi((x_0, x_1), (z_0, z_1))+ \phi((x_0, x_1)* (z_0, z_1), (y_0, y_1) * (z_0,z_1)).\end{eqnarray*}
We have
\begin{eqnarray*}
	{\rm LHS} &=& \psi_0 (x_0, y_0, y_1) + \psi_1 (x_1, y_0, y_1) + \\
	&& \psi_0 (T_0(x_0, y_0, y_1), z_0, z_1) + \psi_1 (T_1(x_1, y_0,y_1), z_0, z_1), 	\\
	{\rm RHS} &=& \psi_0 (x_0, z_0, z_1) + \psi_1 (x_1, z_0, z_1)+\\
	&& \psi_0 (T_0(x_0,z_0, z_1), T_0(y_0, z_0, z_1), T_1(y_1, z_0, z_1)) +\\
	&& \psi_1 (T_1(x_1,z_0, z_1), T_0(y_0, z_0, z_1), T_1(y_1, z_0, z_1)). 	
\end{eqnarray*}
The compatibility  conditions of $\psi_0$ and $\psi_1$  show that LHS and RHS coincide.
\end{proof}

%

%
  
The constructions are summarized as follows.

\begin{proposition}\label{prop:factor}
	It holds that ${\mathcal G }\circ  {\mathcal F} = {\mathcal D}_R $ and 
	${\mathcal F} \circ {\mathcal G} =  {\mathcal D}_T $. 
\end{proposition}

\begin{proof}
Let $(X, *_0, *_1)$ be a set with mutually distributive rack operations.
Let $(X, T)={\mathcal F}(X, *_0, *_1)$.
Then by definition $T(x, y_0, y_1)=(x *_0 y_0 ) *_1 y_1$.
Lemma~\ref{lem:TtoR} implies that $(X \times X, *)={\mathcal G}(X, T, T)$ is a rack, since $T$ is mutually distributive over itself.
One computes
\begin{eqnarray*}
\lefteqn{ {\mathcal G}(X, T, T) = (x_0, x_1)*( y_0, y_1)} \\
&=& (T (x_0, y_0, y_1), T(x_1, y_0, y_1) ) \\
&=&  ((x_0  *_0 y_0 ) *_1 y_1, (x_1  *_0 y_0 ) *_1 y_1 ) \\
&=&  {\mathcal D}_R( X, *_0, *_1) 
\end{eqnarray*}
as desired.

Let $(X, T_0, T_1)$ be a set with mutually distributive ternary rack operations.
Let $(X\times X, *)={\mathcal G}(X, T_0, T_1)$.
Then by definition $(x_0, x_1) * ( y_0, y_1)=( T_0 (x_0, y_0, y_1), T_1( x_1, y_0 ,y_1 ) ) $.
Since $*$  is mutually distributive over itself, 
we have that $(X \times X, T)={\mathcal F}(X \times X, *, * )$ is a rack, as in Definition~\ref{def:bintoter}.
One computes
\begin{eqnarray*}
\lefteqn{ {\mathcal F}(X \times X, *, * ) = T ( (x_0, x_1), ( y_0, y_1), (z_0, z_1)  )   }\\
&=&  [ (x_0, x_1) * ( y_0, y_1) ] * (z_0, z_1) \\
&=& (  T_0 (x_0, y_0, y_1), T_1( x_1, y_0 ,y_1 ) ) * (z_0, z_1) \\ 
&=& ( T_0(  T_0 (x_0, y_0, y_1), z_0, z_1), T_1( T_1( x_1, y_0 ,y_1 ) , z_0, z_1 ) \\
&=& {\mathcal D}_T ( X, T_0, T_1)
\end{eqnarray*}
as desired.
\end{proof}

\begin{proof}[Proof of Theorem~\ref{thm:doubleRcocy}]
Let $( *_0, *_1)$ be  mutually distributive rack operations on $X$.
Let $(X , T )={\mathcal F}(X, *_0, *_1)$.
We have that $(X, T)$ is a ternary rack.
Let $\phi_0, \phi_1$ be mutually distributive rack $2$-cocycles of 
$(X, *_0)$ and $(X, *_1)$, respectively.
Then by Theorem~\ref{thm:TRcocy}, 
$$\psi( x, y_0, y_1 ) := \phi_0 (  x, y_0 ) + \phi_1 (x *_0 y_0 , y_1) $$
is a ternary rack $2$-cocycle of $(X , T)$. 
Since $T$ is compatible over itself, 
\begin{eqnarray*}
\lefteqn{ ({\mathcal G}\circ {\mathcal F})  (X , *_0, *_1) ( (x_0, x_1), (y_0, y_1), (z_0, z_1)) }\\
&=& {\mathcal G} ( X \times X, T, T)  ( (x_0, x_1), (y_0, y_1), (z_0, z_1)) \\
&=& (T (T(x_0, y_0, y_1), z_0, z_1 ), T (T(x_1, y_0, y_1), z_0, z_1 )
\end{eqnarray*}
is a rack operation by Theorem~\ref{thm:dbleT}.
Then Theorem~\ref{thm:TRcocy} applied to $(X \times X, T, T)$ with mutually distributive cocycles 
$(\psi, \psi)$ implies that 
\begin{eqnarray*}
\lefteqn{
\phi( ( x_0, x_1), (y_0, y_1) ) }\\
&=&
\psi ( x_0, y_0, y_1 ) + \psi( x_1, y_0, y_1 ) \\
&=& 
\phi_0 (  x, y_0 ) + \phi_1 (x *_0 y_0 , y_1) +  \phi_0 (  x_1 , y_0 ) + \phi_1 (x_1  *_0 y_0 , y_1) 
\end{eqnarray*}
as desired. To show that the assignment $\Theta(\phi_0,\phi_1) = \psi$ passes to cohomology, it is enough to show that if $(\phi_0,\phi_1) = \delta_L^1 f$, we have that $\Theta (\delta_L^1 f) = \delta^1_R g$, for some $1$-cochain $g$. It is easy to see that the map $g (x_0,x_1) := f(x_0) + f(x_1)$ does indeed serve the purpose.
\end{proof}

\begin{proof}[Proof of Theorem~\ref{thm:dbleTcocy}]
Let $(T_0, T_1)$ be compatible ternary distributive operations on $X$, 
and $(X \times X, *)={\mathcal G}(X, T_0, T_1)$.
By  Lemma~\ref{lem:TtoR}, $(X \times X , *)$ is a rack.
Let $\psi_0, \psi_1$ be compatible ternary $2$-cocycles of 
$(X, T_0)$ and $(X, T_1)$, respectively.
Then by Theorem~\ref{thm:TRcocy}, 
$$\phi( (x_0, x_1), (y_0, y_1) ) := \psi_0 (x_0, y_0, y_1) + \psi_1 (x_1, y_0, y_1) $$
is a rack $2$-cocycle of $(X \times X , *)$. 
Since $*$ is mutually distributive over itself, 
\begin{eqnarray*}
\lefteqn{ ({\mathcal F}\circ {\mathcal G})  (X, T_0, T_1) ( (x_0, x_1), (y_0, y_1), (z_0, z_1)) }\\
&=&  T( (x_0, x_1), (y_0, y_1), (z_0, z_1)) \\
&=&   [  (x_0, x_1) *  (y_0, y_1) ] * (z_0, z_1)
\end{eqnarray*}
is a ternary rack operation by Lemma~\ref{lem:RtoTfunctor}. 
Then Theorem~\ref{lem:labcohmaps} applied to $(X \times X, *, *)$ with mutually distributive cocycles 
$(\phi, \phi)$ implies that 
\begin{eqnarray*}
\lefteqn{
\psi( (x_0, x_1), (y_0, y_1), (z_0, z_1)) }\\
&=&
\phi(  (x_0, x_1), (y_0, y_1) ) + \phi( (x_0, x_1) *  (y_0, y_1),  (z_0, z_1)  ) \\
&=& 
\phi(  (x_0, x_1), (y_0, y_1) ) + \phi( ( T_0 (x_0, y_0, y_1),  T_1 ( x_1, y_0, y_1) ), (z_0, z_1)  ) \\
&=&  \psi_0 (x_0, y_0, y_1) + \psi_1 (x_1, y_0, y_1) \\
& & +  \ \psi_0( ( T_0 (x_0, y_0, y_1), z_0, z_1)  + \psi_1(   T_1 ( x_1, y_0, y_1) ), z_0, z_1  )
\end{eqnarray*}
as desired.
\end{proof}

\section{Internalization of higher order self-distributivity}\label{Inter} 

We begin this section with the definition of $n$-ary self-distributive object in a symmetric monoidal category, providing therefore a higher arity version of the work in \cite{CCES}.  We will use the symbol $\boxtimes$ to indicate the tensor product in the symmetric monoidal category $\mathcal{C}$, not to confuse the general setting with the standard tensor product in vector spaces, to be found in the examples. We remind the reader first, that a symmetric monoidal category is a monoidal category $\mathcal{C}$ together with a family of isomorphisms $\tau_{X,Y}:X\boxtimes Y\longrightarrow Y\boxtimes X$, natural in $X$ and $Y$, satisfying the following conditions (Section 11 in \cite{MacL}). The hexagon:
$$
\begin{tikzcd}
&	X\boxtimes (Y\boxtimes Z)\arrow[dr,"\tau{X,Y\boxtimes Z}"] &\\
(X\boxtimes Y)\boxtimes Z \arrow[ur,"\alpha_{X, Y, Z}"]\arrow[d,swap,"\tau_{X,Y}\boxtimes \mathbbm{1}"] &    & (Y\boxtimes Z)\boxtimes X\arrow[d,"\alpha_{Y,Z,X}"]\\
(Y\boxtimes X)\boxtimes Z\arrow[dr,swap,"\alpha_{Y,X,Z}"] &     &  Y\boxtimes(Z\boxtimes X)\\
   & Y\boxtimes (X\boxtimes Z)\arrow[ur,swap,"\mathbbm{1}\boxtimes \tau_{X,Z}"]    &
	\end{tikzcd}
$$
is commutative for all objects $X$,$Y$ and $Z$ in $\mathcal{C}$, where $\alpha_{X, Y, Z}$ indicates the associator of the monoidal category. We further have the following identity for all objects $X$ and $Y$: 
$$
\tau_{Y,X}\tau_{X,Y} = \mathbbm{1}_{X\boxtimes Y}.
$$
 For the sake of simplicity, we  work on a strict symmetric monoidal category for the rest of the paper
 and therefore do not keep track of the bracketing. We recall also that a comonoid in a symmetric monoidal category is an object $X\in \mathcal{C}$ endowed with morphisms $\Delta:X \longrightarrow X\boxtimes X$, called comultiplication or diagonal, and $\epsilon:X \longrightarrow I$, called counit, where $I$ is the unit object of the monoidal category. The comultiplication and the counit satisfy the usual coherence diagrams analogous to the coalgebra axioms. In virtue of the coassociative axiom we can inductively define an $n$-diagonal $\Delta_n:X \longrightarrow X^{\boxtimes n}$ by the assignment: $\Delta_n = (\Delta\boxtimes \mathbbm{1})\Delta_{n-1}$, for all $n\in\mathbb{N}$. Let us define the isomorphism $\tau_{i,i+1}: X^{\boxtimes n} \longrightarrow X^{\boxtimes n}$ as $\tau_{i,i+1} = \mathbbm{1}^{\boxtimes (i-1)}\boxtimes \tau_{X,X}\boxtimes \mathbbm{1}^{\boxtimes (n-i-1)}$. It is easy to verify that the morphisms $\tau_{i,i+1}$ satisfy the relations of the transposition $(i,i+1)$ in $\mathbbm{S}_n$, the symmetric group on $n$ letters. We therefore obtain, for every object $X$, an action of $\mathbbm{S}_n$ on $X^{\boxtimes n}$, by mapping $(i,i+1)$ to $\tau_{i,i+1}$, and extending to a homomorphism of groups between $\mathbbm{S}_n$ and ${\rm Aut}(X^{\boxtimes n})$, the automorphism group of $X^{\boxtimes n}$. In particular we will make use of the automorphism of $X^{\boxtimes n^2}$, corresponding to the permutation 
 \begin{eqnarray*}
\shuffle_n &=&  (2,n+1)(3,2n+1)\cdots (n,(n-1)n+1) \\
& & (n+3,2n+2)(n+4,3n+2)\cdots (2n,(n-1)n+2) \\
& & \cdots ((n-2)n+n,(n-1)n+n-1). 
\end{eqnarray*}
We are ready now to define $n$-ary self-distributive objects in a symmetric monoidal category $\mathcal{C}$. 

\begin{definition}\label{def:n-arysd}
	\rm{
	An $n$-ary self-distributive object in a symmetric monoidal category $\mathcal{C}$ is a pair $(X,W)$, where $X$ is a comonoid object in $\mathcal{C}$ and $W:X^{\boxtimes n} \longrightarrow X$ is a morphism making the following diagram commute:
\begin{center}
	\begin{tikzcd}
		&X^{\boxtimes n^2}\arrow{dl}[swap]{\shuffle_n} & &X^{\boxtimes (2n-1)}\arrow{ll}[swap]{\mathbbm 1^{\boxtimes n}\boxtimes \Delta_n^{\boxtimes (n-1)}}\arrow {rd}{W\boxtimes \mathbbm{1}^{\boxtimes (n-1)}} &  \\
		 X^{\boxtimes n^2}\arrow{dd}[swap]{W\boxtimes \cdots \boxtimes W} & & & & X^{\boxtimes n}\arrow{dd}{W}\\
		 & & & &\\
		 X^{\boxtimes n}\arrow{rrrr}[swap]{W}& & & &X 
		\end{tikzcd} 
	\end{center}
}
	\end{definition}


 \begin{example}
 	\rm{
 	Clearly, any $n$-ary rack is an $n$-ary self-distributive object in the symmetric monoidal category of sets, with $\tau$ and $\Delta$ defined in the obvious way. 
 }
 \end{example}

In the rest of this section we will make use of Sweedler notation in the following form: $\Delta (x) = x^{(1)}\otimes x^{(2)}$. 

\begin{example}\label{qheap}
	\rm{
Let $H$ be an involutive Hopf algebra, i.e. $S^2 = \mathbbm{1}$. Define a ternary operation $T:H\otimes H\otimes H \longrightarrow H$ by the assignment $T(x\otimes y \otimes z) = xS(y)z$, extended by linearity, where we use juxtaposition as a shorthand to indicate the multiplication $\mu$ of $H$ and $S$ is the antipode. By direct computation on tensor monomials we obtain, for the left hand side of ternary self-distributivity:
\begin{eqnarray*}
	\lefteqn{T(T(x\otimes y \otimes z)\otimes u \otimes z)}\\
	&=& T(xS(y)z\otimes u\otimes v)\\
	&=& xS(y)zS(u)v.
	\end{eqnarray*}
The right hand side is: 
\begin{eqnarray*}
\lefteqn{TT^{\otimes 3} \shuffle_3 (\mathbbm{1}^{\otimes 3}\otimes (\Delta\otimes \mathbbm{1})\Delta \otimes (\Delta\otimes \mathbbm{1})\Delta)(x\otimes y \otimes z \otimes u \otimes v)}\\
&=& TT^{\otimes 3}((x\otimes u^{(11)} \otimes v^{(11)}\otimes (y\otimes u^{(12)}\otimes v^{(12)})\otimes (z\otimes u^{(2)}\otimes v^{(2)}))\\
&=& T(xS(u^{(11)})v^{(11)})\otimes yS(u^{(12)})v^{(12)}\otimes zS(u^{(2)})v^{(2)})\\
&=& xS(u^{(11)})v^{(11)}S(yS(u^{(12)})v^{(12)})zS(u^{(2)})v^{(2)}\\
&=& xS(u^{(11)})v^{(11)}S(v^{(12)})S^2(u^{(12)})S(y)zS(u^{(2)})v^{(2)}\\
&=& xS(u^{(11)})\epsilon(v^{(1)}\cdot 1)S^2(u^{(12)})S(y)zS(u^{(2)})v^{(2)}\\
&=& xS(\epsilon(u^{(1)})\cdot 1)S(y)z S(u^{(2)})\epsilon(v^{(1)})v^{(2)}\\
&=& xS(y)zS(\epsilon(u^{(1)})u^{(2)})v\\
&=& xS(y)zS(u)v.
\end{eqnarray*}
Note that we have used the fact that $H$ is involutive in the sixth equality, to obtain $S(u^{(12)})u^{(11)} = u^{(1)}$. This ternary structure is the Hopf algebra analogue of the heap operation in group theory, which is known to be ternary self-distributive. We also observe that $H$ being involutive is a parallel to the operation of taking inverses, obviously involutive as well.
} 
\end{example}
\begin{figure}[htb]\label{fig:catsd}
	\begin{center}
		\includegraphics[width=2.5in]{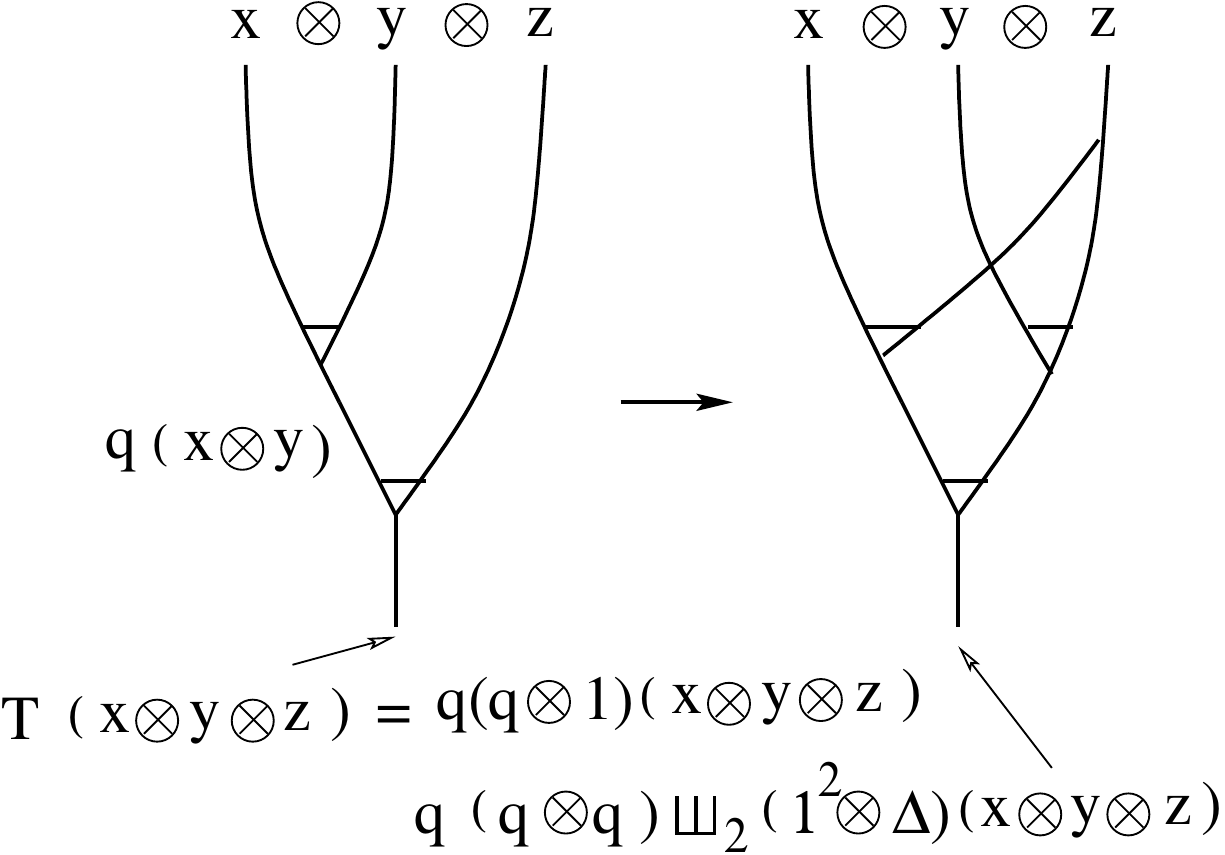}\\
		\caption{Diagrammatic  representation of categorical distributivity}\label{fig:q}
	\end{center}
\end{figure}

In Figure~\ref{fig:q}, a diagrammatic representation of categorical distributivity is depicted.
It is read from top to bottom, where the top 3 end points of both sides represent $x \otimes y \otimes z$,
a trivalent vertex with a small triangle represents a self-distributive morphism $q: X\otimes X \rightarrow X$,
and the left-hand side represents $T=q(q\otimes {\mathbbm 1})$.

Given a symmetric monoidal category $\mathcal{C}$, we define categories $n\mathcal{SD}$, for each $n\in \mathbb{N}$, as follows. The objects are $n$-ary self-distributive objects in $\mathcal{C}$, as in Definition \ref{def:n-arysd}. Given two objects $(X,q)$ and $(X',q')$, we define the morphism class between them to be the class of morphism $f:X\longrightarrow X'$ in $\mathcal{C}$ , such that $f\circ q = q'\circ f^{\boxtimes n}$. In particular we define $\mathcal{BSD} = 2\mathcal{SD}$ and $\mathcal{TSD} = 3\mathcal{SD}$, $\mathcal{B}$ and $\mathcal{T}$ standing for binary and ternary, respectively.

We will make use of the following results in Theorem \ref{co-dbl}. 
\begin{lemma}\label{lem:switchco}
	Let $\mathcal{C}$ be a strict symmetric monoidal category. Suppose $(X,\Delta, \epsilon)$ is a comonoid in $\mathcal{C}$. Then the switching morphism and the comultiplication commute. More specifically, we have: $\Delta \boxtimes \mathbbm{1} \circ \tau_{X,Y} = \tau_{X,Y^{\boxtimes 2}}\circ \mathbbm{1} \boxtimes \Delta$. 
\end{lemma}
This lemma is represented in Figure \ref{fig:tauDQ} (A) below.
\begin{proof}
	Consider the following diagram:
$$
\begin{tikzcd}
X\boxtimes Y \arrow[r,"\tau_{X,Y}"]\arrow[dr,swap,"\mathbbm{1}\boxtimes \Delta"] & Y\boxtimes X \arrow[r,"\Delta \boxtimes \mathbbm{1}"] & Y^{\boxtimes 2} \boxtimes X\\
& X\boxtimes Y^{\boxtimes 2}\arrow[r,swap,"\tau_{X,Y} \boxtimes\mathbbm{1}"]\arrow[ur,swap,"\tau_{X,Y^{\boxtimes 2}}"] & Y\boxtimes X\boxtimes Y \arrow[u,swap,"\mathbbm{1}\boxtimes \tau_{X,Y}"]
\end{tikzcd}
$$
		The outmost diagram commutes by naturality of switching map $\tau_{X,Y}$ with respect to $X$ and $Y$. The lower right triangle commutes by the hexagon axiom:
		$$
		\begin{tikzcd}
		&	X\boxtimes (Y\boxtimes Y)\arrow[dr,equal]\arrow[dl,swap,"\tau{X,Y^{\boxtimes 2}}"] &\\
		 (Y\boxtimes Y)\boxtimes X\arrow[d,equal] &    & (X\boxtimes Y)\boxtimes Y\arrow[d,"\tau_{X,Y} \boxtimes \mathbbm{1}"]\\
		Y\boxtimes (Y\boxtimes X) &     &  (Y\boxtimes X)\boxtimes Y\arrow[dl,equal]\\
		& Y\boxtimes (X\boxtimes Y)\arrow[ul,"\mathbbm{1}\boxtimes \tau_{X,Y}"]   &
		\end{tikzcd}
		$$
		The assertion now follows.
	\end{proof}
\begin{lemma}\label{lem:switchq}
	Let $(X,q)$ be a binary self-distributive object in a strict symmetric monoidal category $\mathcal{C}$. Then the switching morphism and the self-distributive operation commute. More specifically, we have: $\tau_{X,Y}\circ q\boxtimes \mathbbm{1} = \mathbbm{1}\boxtimes q\circ \tau_{X^{\boxtimes 2},Y}$. 
	\end{lemma}
This lemma is represented in Figure \ref{fig:tauDQ} (B) below.
 \begin{proof}
 	Similar to Lemma \ref{lem:switchco} and left to the reader. 
 	\end{proof}
 \begin{figure}[htb]
 	\begin{center}
 		\includegraphics[width=2.5in]{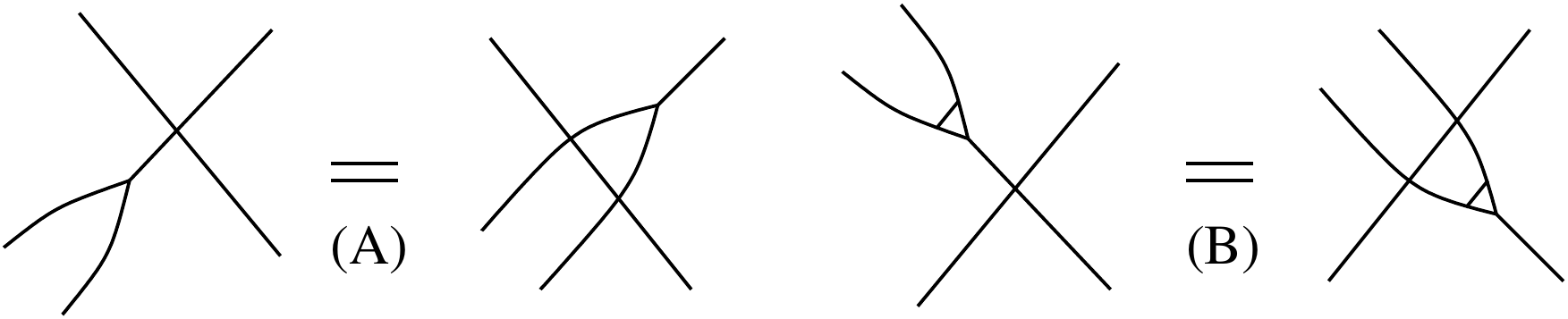}\\
 		\caption{The switching morphism commutes with comultiplication and binary self-distributive operation}\label{fig:tauDQ}
 	\end{center}
 \end{figure}
In general, the following result is useful to produce ternary self-distributive objects in  the category of vector spaces, starting from binary self-distributive objects (see also \cite{CCES}). Compare it to the construction of Section \ref{sec:nary}. 
\begin{theorem}\label{co-dbl}
	Let $(X,\Delta)$ be a comonoid in a (strict) symmetric monoidal category $\mathcal{C}$ (e.g. a coalgebra in the category of vector spaces). Let $q:X\boxtimes X\longrightarrow X$ be a morphism such that $(X,q)$ is a binary self-distributive object in $\mathcal{C}$. Then the pair $(X, T)$, where $T = q(q\boxtimes \mathbbm{1})$, defines a ternary self-distributive object in $\mathcal{C}$. The construction defines a functor $\mathcal{F}:\mathcal{BSD}\rightarrow \mathcal{TSD}$.
	\end{theorem}
\begin{proof} 
	We define $\mathcal{F}$ on objects as $\mathcal{F}(X,q) = (X, T)$ and as the identity on morphisms. To show that the map $T = q(q\boxtimes \mathbbm{1})$ is ternary self-distributive, we can proceed as in Figure \ref{fig:qq}. 
	In the left column of the figure, the part of the diagram representing each $T = q(q\boxtimes \mathbbm{1})$ are
	indicated  by dotted circles.
	 At each step we are using the definition of $T$, the binary self-distributivity of $q$ and Lemmas \ref{lem:switchco} and \ref{lem:switchq}. If $f:(X,q) \longrightarrow (Y,q')$ is a morphism in $\mathcal{BSD}$, we can show that $f$ is also a morphism in $\mathcal{TSD}$ between $(X,T=q(q\boxtimes \mathbbm{1}))$ and $(Y, T' = q'(q'\boxtimes \mathbbm{1}))$ via the following diagram: 
	$$	\begin{tikzcd}
	X\boxtimes X\boxtimes X  \arrow[d,"f\boxtimes f\boxtimes f"] \arrow [r,"q\boxtimes \mathbbm{1}"] & X\boxtimes X \arrow[d,"f\boxtimes f"] \arrow [r,"q"]& X\arrow[d,"f"]\\
	Y\boxtimes Y\boxtimes Y \arrow[r,"q'\boxtimes \mathbbm{1}"] & Y\boxtimes Y \arrow[r,"q'"] &Y
	\end{tikzcd}$$
	where the commutativity of the left and right squares is just a restatement of the fact that $q$ is a morphism in $\mathcal{BSD}$. The consequent commutativity of the outer rectangle means that $f$ is a morphism in $\mathcal{TSD}$ as well. It is also clear that $\mathcal{F}$ preserves composition of morphisms. 
	\end{proof}

\begin{figure}[htb]
	\begin{center}
		\includegraphics[width=3in]{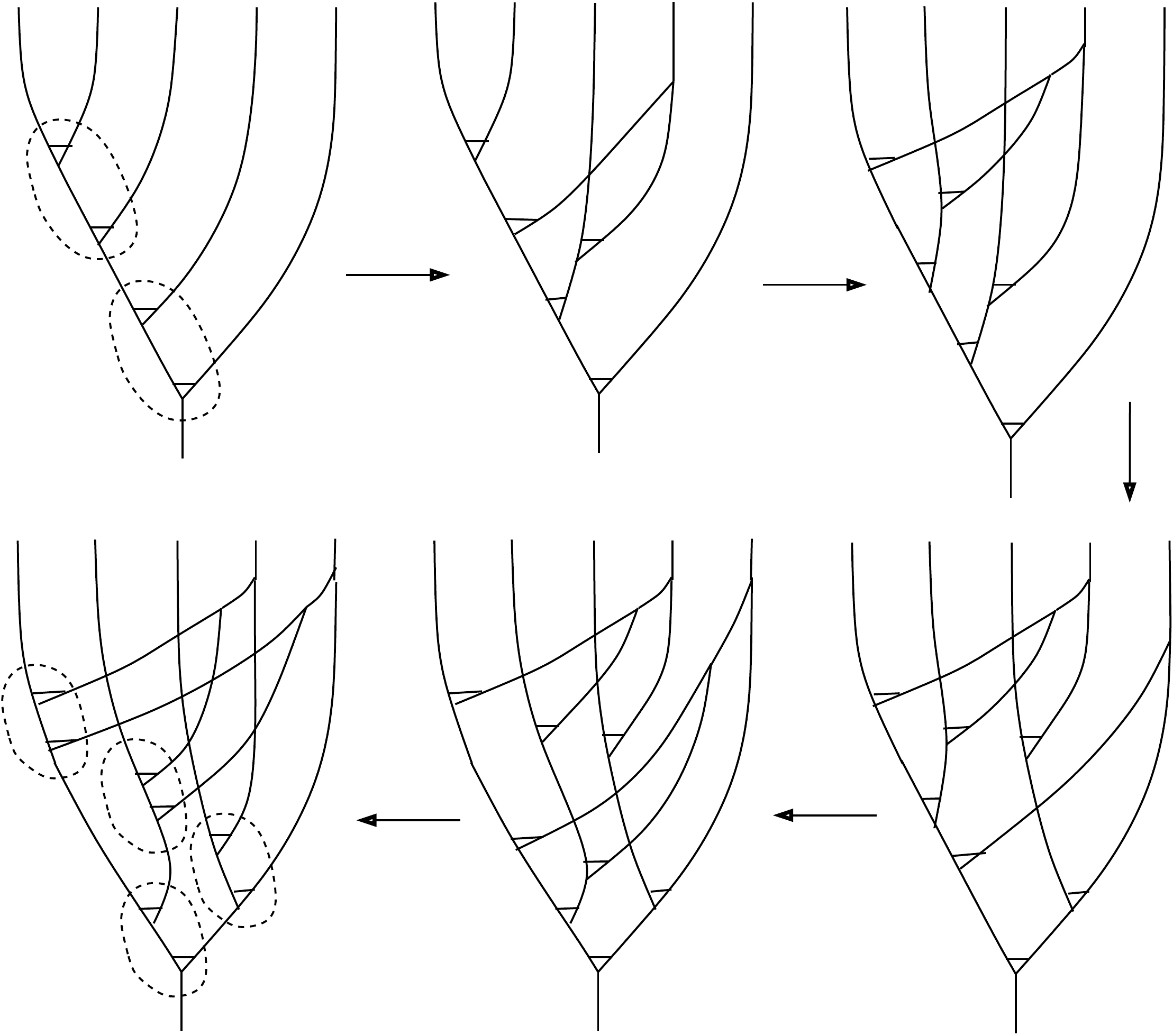}\\
		\caption{Diagrammatic proof of doubling procedure}\label{fig:qq}
	\end{center}
\end{figure}

The following is a rephrased version of Lemma 3.3 in \cite{CCES}, adapted to our language in the present article. 
\begin{lemma}\label{lem:CCES}
	Let $L$ be a Lie algebra over a ground field $\mathbbm {k}$. Define $X = \mathbbm{k}\oplus L$ and endow it with a comultiplication $\Delta$, defined by $(a,x) \mapsto (a,x)\otimes (1,0) + (1,0)\otimes (0,x)$, and a counit $\epsilon$, defined by $(a,x) \mapsto a$. Then $(X,\Delta , \epsilon)$ is a comonoid in the symmetric monoidal category of vector spaces. The morphism $q:X\otimes X \longrightarrow X$ defined by $(a,x)\otimes (b,y)\mapsto (ab, bx + [x,y])$ turns $X$ into a binary self-distributive object. 
\end{lemma}
\begin{proof}
	By direct computation making use of the Jacobi identity. This is done explicitly in Lemma 3.3 in \cite{CCES}. 
	\end{proof}
\begin{example}\label{Lie}
	\rm{
 Let $L$ be a Lie algebra and let $X = \mathbbm{k}\oplus L$ be as in Lemma \ref{lem:CCES}. The map $T: X\otimes X\otimes X \longrightarrow X$ defined by
 $$(a,x)\otimes (b,y)\otimes (c,z) \longmapsto (abc, bcx + c[x,y] + b[x,z] + [[x,y],z]),$$
  and extended by linearity, is such that $(X,T)$ is a tenrary self-distributive object in the category of vector spaces by an easy application of Theorem \ref{co-dbl}. An explicit, and tedious, computation that shows the self-distributivity of $T$ directly, is postponed to Appendix \ref{Ap:Lie}. 
}
\end{example} 

If $H$ is a Hopf algebra, we can use the adjoint map to produce a ternary self distributive map, as the following example shows:
\begin{example}
	\rm{
 The map defined by $T(x\otimes y\otimes z) = S(z^{(1)})S(y^{(1)})xy^{(2)}z^{(2)}$  is ternary self-distributive, as an easy direct computation shows. This is the Hopf algebra analogue of the iterated conjugation quandle.
}
\end{example}

\begin{remark}
	\rm{
	It is possible, a priori, to develop the theory of higher self-distributivity in braided monoidal categories, where the switching morphism satisfies the hexagon axiom but we do not require $\tau_{Y,X}\tau_{X,Y} = \mathbbm{1}_{X\boxtimes Y}$. Similarly as above we have an action of the braid group on $n$ strings on every object $X^{\boxtimes n}$ and the shuffle map $\shuffle_n$ takes now into account over passing and under passing of the strings. }
\end{remark}

%

\appendix  

\section{Example \ref{Lie} revisited}\label{Ap:Lie}

In this appendix we explicitly show that the map in Example \ref{Lie} is indeed self-distributive. Each equality is obtained by applying the Jacobi identity as in the proof of Lemma 3.3 in \cite{CCES}. In fact, each step corresponds to one of the diagrams in the proof of Theorem \ref{co-dbl} (cf. figure \ref{fig:qq}). Recall also the definition of the diagonal $\Delta$, from Lemma \ref{lem:CCES}, and the inductive definition for $\Delta_3$ at the beginning of Section \ref{Inter}. Explicitly, we have for $\Delta_3$:
$$
\Delta_3 (a,x) = (a,x)\otimes (1,0)\otimes (1,0) + (1,0)\otimes (0,x)\otimes (1,0) + (1,0)\otimes (1,0)\otimes (0,x).
$$
To make the steps easier for the reader, we declare the terms that are going to be replaced according to the Jacobi identity, and underline the replacing terms in the subsequent equality. We obtain therefore: 
\begin{eqnarray*}
	\lefteqn{T(T((a,x)\otimes (b_0,y_0)\otimes (b_1,y_1))\otimes (c_0,z_0) \otimes (c_1,z_1))}\\
	&=& (ab_0b_1c_0c_1, b_0b_1c_0c_1 x + b_1c_0c_1 [x,y_0] + b_0c_0c_1 [x,y_1]\\ 
	&& + c_0c_1 [[x,y_0],y_1] + b_0b_1c_1 [x,z_0] + b_1c_1 [[x,y_0],z_0] \\
	&& + b_0c_1[[x,y_1],z_0] + c_1[[[x,y_0],y_1],z_0] + b_0b_1c_0 [x,z_1]\\
	&& + b_1c_0 [[x,y_0],z_1] + b_0c_0[[x,y_1],z_1] + c_0[[[x,y_0],y_1],z_1]\\
	&& + b_0b_1[[x,z_0],z_1] + b_1[[[x,y_0],z_0],z_1] + b_0[[[x,y_1],z_0],z_1]\\
	&& + [[[[x,y_0],y_1],z_0],z_1]).
\end{eqnarray*}
\noindent Applying the Jacobi identity to the terms $b_0c_1[[x,y_1],z_0]$, $c_1[[[x,y_0],y_1],z_0]$, $b_0[[[x,y_1],z_0],z_1]$ and $[[[[x,y_0],y_1],z_0],z_1]$ we obtain:
\begin{eqnarray*}
	&=& (ab_0b_1c_0c_1, b_0b_1c_0c_1x + b_1c_0c_1[x,y_0] + b_0b_1c_1[x,z_0]\\
	&& + b_1c_1[[x,y_0],z_0] + b_0c_0c_1[x,y_1] + c_0c_1[[x,y_0],y_1]\\
	&& + \underline{b_0c_1[[x,z_0],y_1]} + \underline{c_1[[[x,y_0],z_0],y_1]} + \underline{b_0c_1[x,[y_1,z_0]]}\\
	&& + \underline{c_1[[x,y_0],[y_1,z_0]]}  + b_0b_1c_0[x,z_1] + b_1c_0[[x,y_0],z_1]\\
	&& + b_0b_1[[x,z_0],z_1] + b_1[[[x,y_0],z_0],z_1]+ b_0c_0[[x,y_1],z_1]\\
	&& + c_0[[[x,y_0],y_1],z_1] + \underline{b_0[[[x,z_0],y_1],z_1]} + \underline{[[[[x,y_0],z_0],y_1],z_1]}\\  
	&& + \underline{b_0[[x,[y_1,z_0]],z_1]} + \underline{[[[x,y_0],[y_1,z_0]],z_1]}).
\end{eqnarray*}
We now apply the Jacoby identity to the term $b_1c_1[[x,y_0],z_0]$, $b_1[[[x,y_0],z_0],z_1]$, $c_1[[[x,y_0],z_0],y_1]$ and $[[[[x,y_0],z_0],y_1],z_1]$ to obtain:
\begin{eqnarray*}
	&=& (ab_0b_1c_0c_1,  b_0b_1c_0c_1x + b_0b_1c_1[x,z_0] + b_1c_0c_1[x,y_0]\\
	&& + \underline{b_1c_1[[x,z_0],y_0]} + b_0c_0c_1 [x,y_1] + b_0c_1[[x,z_0],y_1]\\
	&& + c_0c_1[[x,y_0],y_1]] + \underline{c_1[[[x,z_0],y_0],y_1]} + \underline{b_1c_1[x,[y_0,z_0]]}\\
	&& + \underline{c_1[[x,[y_0,z_0]],y_1]} + b_0c_1[x,[y_1,z_0]] + c_1[[x,y_0],[y_1,z_0]]\\
	&& + b_0b_1c_0[x,z_1] + b_0b_1[[x,z_0],z_1] + b_1c_0[[x,y_0],z_1]\\
	&& + \underline{b_1[[[x,z_0],y_0],z_1]} + b_0c_0 [[x,y_1],z_1] + b_0[[[x,z_0],y_1],z_1]\\
	&& + c_0[[[x,y_0],y_1],z_1] + \underline{[[[[x,z_0],y_0],y_1],z_1]} + \underline{b_1[[x,[y_0,z_0]],z_1]}\\
	&& + \underline{[[[x,[y_0,z_0]],y_1],z_1]} + b_0[[x,[y_1,z_0]],z_1] + [[[x,y_0],[y_1,z_0]],z_1] ).
\end{eqnarray*}
Next, we use the Jacoby identity on the terms $b_0c_0[[x,y_1],z_1]$, $b_0[[[x,z_0],y_1],z_1]$, $b_0[[x,[y_1,z_0]],z_1]$, $c_0[[[x,y_0],y_1],z_1]$, $[[[[x,z_0],y_0],y_1],z_1]$, $[[[x,[y_0,z_0]],y_1],z_1]$ and $[[[x,y_0],[y_1,z_0]],z_1]$. 
\begin{eqnarray*}
	&=&(ab_0b_1c_0c_1, b_0b_1c_0c_1x + b_0b_1c_1[x,z_0] + b_1c_0c_1[x,y_0]\\
	&& + b_1c_1[[x,z_0],y_0] + b_0b_1c_0[x,z_1] + b_0b_1[[x,z_0],z_1]\\
	&& + b_1c_0 [[x,y_0],z_1] + b_1[[[x,z_0],y_0],z_1] + b_0c_0c_1[x,y_1]\\
	&& + b_0c_1[[x,z_0],y_1] + c_0c_1[[x,y_0],y_1] + c_1[[x,[y_0,z_0]],y_1]\\
	&& + b_1c_1[x,[y_0,z_0]] + \underline{b_0c_0[[x,z_1],y_1]} + \underline{b_0[[[x,z_0],z_1],y_1]}\\
	&& + \underline{c_0[[[x,y_0],z_1],y_1]} + b_1[[x,[y_0,z_0]],z_1] + \underline{[[[[x,z_0],y_0],z_1],y_1]}\\
	&& + \underline{[[[x,[y_0,z_0]],z_1],y_1]} + b_0c_1[x,[y_1,z_0]] + c_1[[x,y_0],[y_1,z_0]]\\
	&& + \underline{b_0[[x,z_1],[y_1,z_0]]} + \underline{[[[x,y_0],z_1],[y_1,z_0]]} +\underline{ b_0c_0[x,[y_1,z_1]]}\\
	&& + \underline{b_0[[x,z_0],[y_1,z_1]]} + c_1[[[x,z_0],y_0],y_1] + \underline{c_0[[x,y_0],[y_1,z_1]]}\\ &&+\underline{[[[x,z_0],y_0],[y_1,z_1]]} + \underline{[[x,[y_0,z_0]],[y_1,z_1]]} + \underline{b_0[x,[[y_1,z_0],z_1]]}\\
	&& + \underline{[[x,y_0],[[y_1,z_0],z_1]]}),
\end{eqnarray*}
\begin{sloppypar}\noindent
	Lastly, making use of the Jacobi identity on the terms $b_1c_0 [[x,y_0],z_1]$, $b_1[[[x,z_0],y_0],z_1]$, $c_0[[[x,y_0],z_1],y_1]$, $ b_1[[x,[y_0,z_0]],z_1]$, $ [[[[x,z_0],y_0],z_1],y_1]$,  $[[[x,[y_0,z_0]],z_1],y_1]$ and $[[[x,y_0],z_1],[y_1,z_0]]$ we obtain:
	\end{sloppypar}
\begin{eqnarray*}
	&=& (ab_0b_1c_0c_1, b_0b_1c_0c_1 x + b_0b_1c_1 [x,z_0] + b_0b_1c_0[x,z_1]\\
	&& + b_0b_1[[x,z_0],z_1] + b_1c_0c_1[x,y_0] + b_1c_1[[x,z_0],y_0] \\
	&& +\underline{ b_1c_0[[x,z_1],y_0]} + \underline{b_1[[[x,z_0],z_1],y_0]} + b_0c_0c_1 [x,y_1]\\
	&& + b_0c_1 [[x,z_0],y_1] + b_0c_0[[x,z_1],y_1] + b_0[[[x,z_0],z_1],y_1]\\
	&& + c_0c_1[[x,y_0],y_1] + c_1[[[x,z_0],y_0],y_1] + \underline{c_0[[[x,z_1],y_0],y_1]}\\
	&& + \underline{[[[[x,z_0],z_1],y_0],y_1]} + b_0c_1[x,[y_1,z_0]]  + b_0[[x,z_1],[y_1,z_0]]\\
	&& + c_1[[x,y_0],[y_1,z_0]] + \underline{[[[x,z_1],y_0],[y_1,z_0]]} + b_1c_1[x,[y_0,z_0]]\\
	&& + \underline{b_1[[x,z_1],[y_0,z_0]]}+ c_1[[x,[y_0,z_0]],y_1] + \underline{[[[x,z_1],[y_0,z_0]],y_1]}\\
	&& + \underline{b_1c_0[x,[y_0,z_1]]} +\underline{b_1[[x,z_0],[y_0,z_1]]} +  \underline{c_0[[x,[y_0,z_1]],y_1]}\\
	&& + \underline{[[[x,z_0],[y_0,z_1]],y_1]} + \underline{b_1[x,[[y_0,z_0],z_1]]} + \underline{[[x,[[y_0,z_0],z_1]],y_1]}\\ 
	&& + \underline{[[x,[y_0,z_1]],[y_1,z_0]]}+ b_0c_0[x,[y_1,z_1]] + b_0[[x,z_0],[y_1,z_1]]\\
	&& + c_0[[x,y_0],[y_1,z_1]] + [[[x,z_0],y_0],[y_1,z_1]] + [[x,[y_0,z_0]],[y_1,z_1]]\\
	&& + b_0[x,[[y_1,z_0],z_1]] + [[x,y_0],[[y_1,z_0],z_1]]).
\end{eqnarray*}
This last term can be seen to coincide with the right-hand side of the self-distributivity equation:
$$
T(T^{\otimes 3})\shuffle_3 (\mathbbm{1}^3\otimes \Delta_3^{\otimes 2})((a,x)\otimes (b_0,y_0)\otimes (b_1,y_1))\otimes (c_0,z_0) \otimes (c_1,z_1)).
$$
It follows therefore, that the map $T$ turns $X$ into a ternary self-distributive object in the category of vector spaces. \\

\section{Augmented ternary racks for sets and Hopf algebras}\label{augmented}

It is of an independent interest how the concept of augmented racks generalize to ternary racks
for both sets and monoidal categories in general.
In this section we propose such  generalization
and provide key motivational examples in heaps and Hopf algebras.

An {\it augmented rack}~\cite{FR} $(X, G)$ is a set $X$ with a right group action by  a group $G$ 
and a map $p: X \rightarrow G$ satisfying the identity
$p (x \cdot g) = g^{-1} p(x) g$ for all $x \in X$, $g \in G$.
An augmented rack has a rack operation defined by $x*y=x \cdot  p(y)$ for $x, y \in X$.
The following definition can be considered a ternary analogue of an augmented rack \cite{FR}.

\begin{definition}
	{\rm
		Let $X$ be a set  with a right $G$-action 
		denoted by $X \times G \ni (x, g) \mapsto x \cdot g \in  X$.
		Let $G$ act on the right of $X \times X$ diagonally, 
		$(y_0, y_1) \cdot g=(y_0 \cdot g, y_1 \cdot g)$ for $y_0, y_1 \in X$ and $g \in G$.
		An {\it (double) augmentation} of $X$ is a map $p: X \times X \rightarrow G$ 
		satisfying the condition 
		$$ p(  (y_0, y_1)\cdot g ) = g^{-1} p( ( y_0, y_1 ) ) g $$
		for all $y_0, y_1 \in X$ and $g\in G$. 
	}
\end{definition}

The following is a direct analogue of binary augmented rack
and, therefore, the proof is omitted. 

\begin{lemma}\label{lem:ternaryaug}
	Let $X$ be a set with an augmentation
	$p: X \times X \rightarrow G$. 
	Then the ternary operation $T: X^3 \rightarrow X$ defined by 
	$$ T(x, y_0, y_1) := x \cdot p( ( y_0, y_1) ) $$
	is ternary self-distributive. 
\end{lemma}

\begin{definition}
	{\rm
		Let $X$ be a set with an augmentation $p: X^2 \rightarrow G$ 
		and $T$ be a ternary operation defined  in Lemma~\ref{lem:ternaryaug}.
		Then $(X, T)$ is called an {\it augmented ternary shelf}. 
	}
	\end{definition}

\begin{example}
{\rm
Heaps can be endowed with augmentation as follows.
Let $G$ be a group, with the TSD operation $T(x,y,z)=xy^{-1}z$.
Consider the right multiplication as the right action of $G$ on itself.
Then consider the diagonal right action of $G$ on $G \times G$, by 
$(x, y) \cdot z= (x\cdot z, y \cdot z)$.
Let $p:   G \times G \rightarrow G $ be defined by $p(y, z)=y^{-1} z$.
Then we readily check the condition 
$$ p( ( y_0, y_1) \cdot g )= p ( ( y_0 \cdot g, y_1 \cdot g))=(y_0 g)^{-1} (y_1 g)= g^{-1} y_0^{-1} y_1 g = g^{-1} p((y_0, y_1)) g.$$
}
\end{example}

Although more study on augmented ternary racks are desirable, we focus on the following 
further generalization to Hopf algebras. The point of interest is that the comultiplication plays the role of the diagonal map.

\begin{definition}
	\rm{
Let $X$ be a coalgebra, and let $H$ be a Hopf algebra such that $X$ is a right $H$-module, therefore 
$X^{\otimes 2}$ is also a right $H$-module via the comultiplication in $H$. The map of coalgebras $p: X^{\otimes 2} \longrightarrow H$ is a ternary augmented shelf if, for all $z\in X^{\otimes 2}$ and $g\in H$, we have: 
\begin{center}
$p(z\cdot \Delta(g)) = S(g^{(1)})p(z)g^{(2)}$. 
\end{center}
}
	\end{definition}
This axiom is depicted diagrammatically in Figure \ref{fig:aug}, where solid lines refer to $X$, and dashed lines refer to $H$. We have used $\Delta$, $m$ and $S$ to indicate comultiplication, multiplication and antipode in the Hopf algebra $H$, while $\mu$ stands for the action of $H$ on $X$.
\begin{figure}[htb]
	\begin{center}
		\includegraphics[width=1.8in]{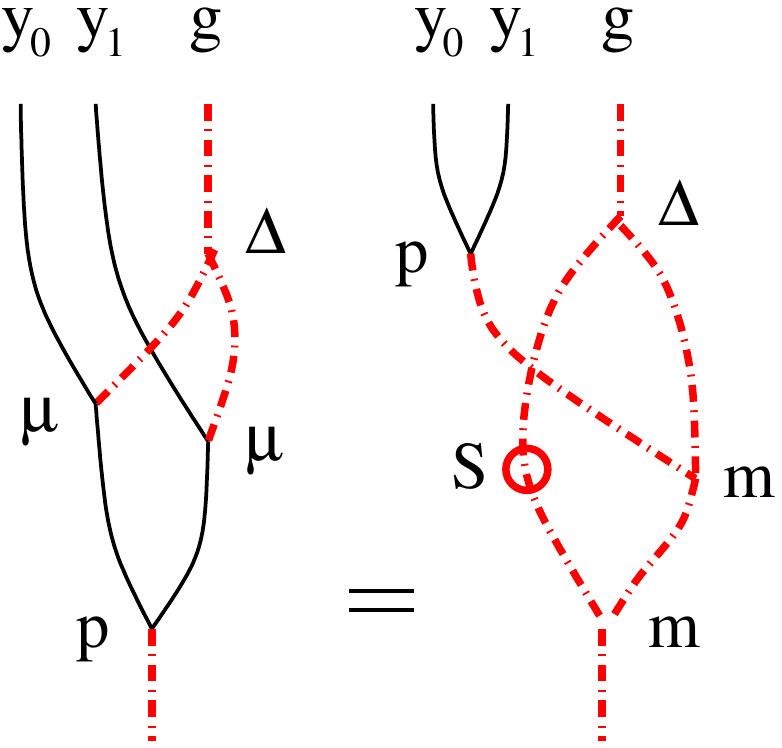}\\
			\caption{Augmented ternary shelf axiom}\label{fig:aug}
	\end{center}
\end{figure}

We have the following result:
\begin{theorem}
	Let $p:X^{\otimes 2} \longrightarrow H$ be a ternary augmented shelf. Then the ternary operation defined on monomials via $x\otimes y\otimes z \longmapsto x\cdot p(y\otimes z)$, and extended by linearity, is self-distributive.
\end{theorem}
\begin{proof}
	By direct computation we have, for the right hand side of self-distributivity axiom: 
		\begin{eqnarray*}
		\lefteqn{TT^{\otimes 3} \shuffle_3 (\mathbbm{1}^{\otimes 3}\otimes (\Delta\otimes \mathbbm{1})\Delta \otimes (\Delta\otimes \mathbbm{1})\Delta)(x\otimes y_0 \otimes y_1 \otimes z_0 \otimes z_1)}\\
		&=& T(x\cdot p(z_0^{(1)}\otimes z_1^{(1)})\otimes y_0\cdot p(z_0^{(2)}\otimes z_1^{(2)})\otimes y_1\cdot p(z_0^{(3)}\otimes z_1^{(3)}))\\
		&=& x\cdot (p(z_0^{(1)}\otimes z_1^{(1)})p(y_0\cdot p(z_0^{(2)}\otimes z_1^{(2)})\otimes y_1\cdot p(z_0^{(3)}\otimes z_1^{(3)})))\\
		&=& x\cdot (p(z_0^{(1)}\otimes z_1^{(1)})(y_0\otimes y_1\cdot \Delta p((z_0\otimes z_1)^{(2)})))\\
		&=& x\cdot (p(z_0\otimes z_1)^{(1)} S(p((z_0\otimes z_1)^{(2)}))p(y_0\otimes y_1)p((z_0\otimes z_1))^{(3)})\\
		&=& x\cdot (\epsilon(p(z_0\otimes z_1)^{(1)})\cdot 1 p(y_0\otimes y_1)  p((z_0\otimes z_1))^{(2)})\\
		&=& x\cdot (p(y_0\otimes y_1)p(z_0\otimes z_1 )),
	\end{eqnarray*}
    where we have used the fact that $p$ is a coalgebra morphism in the third equality, the defining axiom for augmented ternary shelf in the fourth equality, the antipode and the counit axioms to obtain the fifth and sixth equations respectively. It is easy to see that it coincide with the left hand side of self-distributivity.  
	\end{proof}
\begin{example}
	\rm{
  	Let $H$ be an involutive Hopf algebra and let $X = H$. Then, $H$ acts on $X$ via multiplication. Define $p$ to be the map given by $x\otimes y \longmapsto S(x)y$ and extended by linearity. The ternary rack structure obtained is the one in Example \ref{qheap}. A diagrammatic proof that the given $p$ satisfies the augmented ternary rack axiom is shown in Figure \ref{fig:heapaug}.}
  	\end{example}
	
  \begin{figure}[htb]
  	\begin{center}
  		\includegraphics[width=4in]{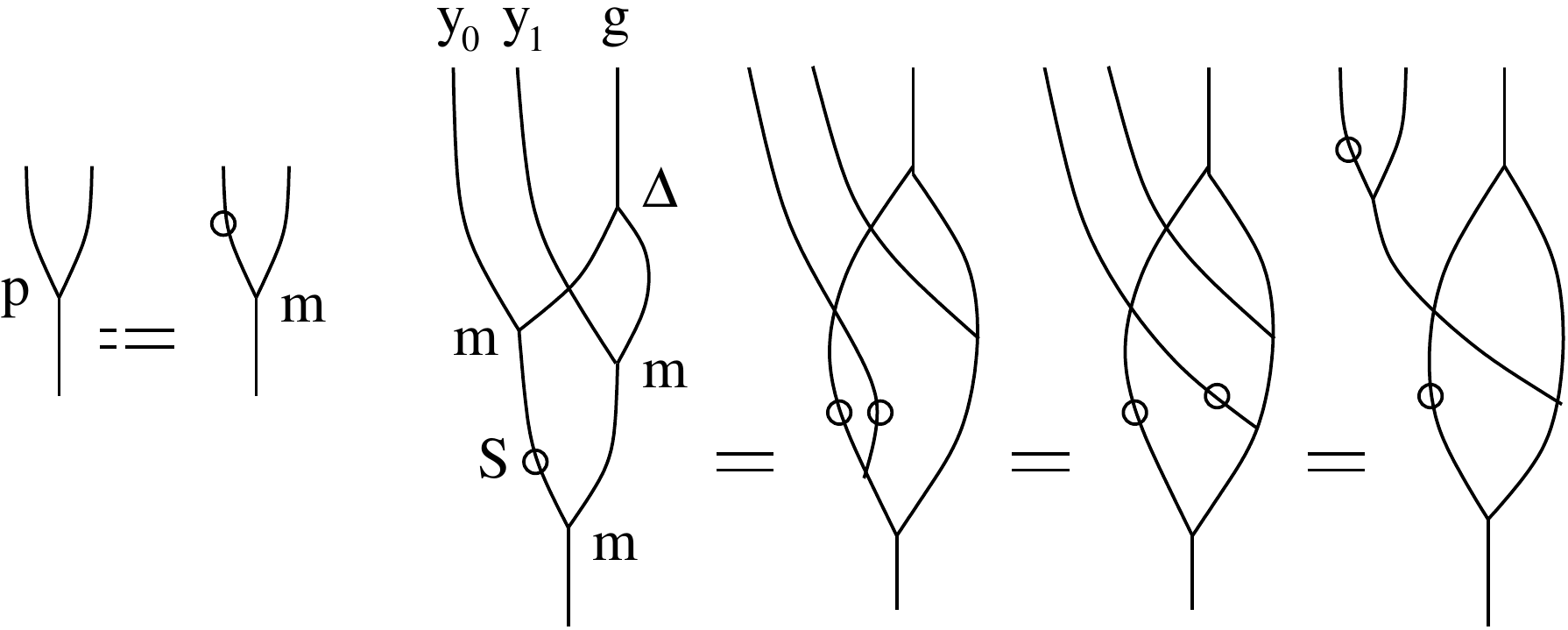}\\
  		\caption{Hopf algebra heap as an augmented ternary shelf}\label{fig:heapaug}
  	\end{center}
  \end{figure}


 \noindent
 {\bf Acknowledgment.}
 M.S. was  supported in part  by NIH R01GM109459 and  NSF DMS-1800443.

%

\end{document}